\documentclass [11pt, a4paper, twoside]{article}

\usepackage {amsmath}
\usepackage{hyperref}
\usepackage{diagrams, a4, indentfirst, latexsym, theorem, QED,
  amssymb, epsfig}

\numberwithin{equation}{section}

\newarrow{Fib}----{>>}
\newarrow{Cof}>--->

\renewcommand{\qedsymbol}{\quad\Box}
\renewcommand{\TheWordProof}{\bf Proof.}
\newcommand\nothing[1]{\relax}
\theoremstyle{change}
\theorembodyfont{\it}
\newtheorem{theorem}{Theorem.}[subsection]
\newtheorem{proposition}[theorem]{Proposition.}
\newtheorem{lemma}[theorem]{Lemma.}
\newtheorem{corollary}[theorem]{Corollary.}
{\theorembodyfont{\upshape}
\newtheorem{definition}[theorem]{Definition.}
\newtheorem{remark}[theorem]{Remark.}
\newtheorem{example}[theorem]{Example.}
\newtheorem{notation}[theorem]{Notation.}
\newtheorem{construction}[theorem]{Construction.}
\newtheorem{numpar}[theorem]{\unskip.}

}

\newcommand{\holim}{\mathrm{holim}}

\newcommand{\bR}{{\mathbb R}}% real numbers
\newcommand{\bZ}{{\mathbb Z}}% integers
\newcommand{\bC}{{\mathbb C}}% complex numbers
\newcommand{\bN}{{\mathbb N}}%  natural numbers
\newcommand{\iso}{\cong}% iso
\newcommand{\tensor}{\otimes}% tensor
\newcommand{\id}{\mathrm{id}}% id
\newcommand{\ie}{{\it i.e.}}% ie.
\newcommand{\resp}{{\it resp.}}% resp.
\newcommand{\pre}{\mathbf{Pre}}
\newcommand{\varpre}{{\mathfrak{Pre}}}
\newcommand{\co}{\mathcal{O}}
\newcommand{\hco}{\hat\co}
\newcommand{\coweq}{\sim\mathrm{co}}
\newcommand{\fan}{\Sigma}
\newcommand{\op}{\mathrm{op}}
\newcommand{\st}{\mathrm{st}}
\newcommand{\qco}{\mathfrak{Qco}}
\date{}

\begin{document}

\title{On the derived category of a regular toric scheme}
\author{Thomas H\"uttemann}
\maketitle

\centerline {\it Queen's University Belfast, Pure Mathematics Research
Centre}
\centerline {\it Belfast BT7~1NN, Northern Ireland, UK}
\centerline {e-mail: \texttt{t.huettemann@qub.ac.uk}}

\vglue 2\bigskipamount \hrule \medskip

{\footnotesize Let $X$ be a quasi-compact scheme, equipped with an
  open covering by affine schemes $U_\sigma = \mathrm{Spec}\,
  A^\sigma$. A quasi-coherent sheaf on~$X$ gives rise, by taking
  sections over the~$U_\sigma$, to a diagram of modules over the
  coordinate rings~$A^\sigma$, indexed by the intersection
  poset~$\fan$ of the covering. If $X$ is a regular toric scheme over
  an arbitrary commutative ring, we prove that the unbounded derived
  category of quasi-coherent sheaves on~$X$ can be obtained from a
  category of $\fan^\op$-diagrams of chain complexes of modules by
  inverting maps which induce homology isomorphisms on hyper-derived
  inverse limits. Moreover, we show that there is a finite set of weak
  generators. If $\fan$ is complete, there is exactly one generator
  for each cone in the fan~$\fan$.

  The approach taken uses the machinery of
  \textsc{Bousfield}-\textsc{Hirschhorn} colocalisation. The first
  step is to characterise colocal objects; these turn out to be
  homotopy sheaves in the sense that chain complexes over different
  open sets~$U_\sigma$ agree on intersections up to
  quasi-isomorphism. In a second step it is shown that the homotopy
  category of homotopy sheaves is equivalent to the derived category
  of~$X$.  \pushright{(\today)}

\smallskip

\noindent
{\it AMS subject classification (2000):\/} primary 18F20, secondary
18E30, 18G55, 55U35

\noindent
{\it Additional keywords:\/} Colocalisation, \textsc{Quillen }model
structures, generators of derived category}
\medskip \hrule \vglue 2\bigskipamount

\section*{Introduction}

A toric scheme $X=X_\fan$ over a commutative ring~$A$ comes equipped
with a preferred covering by open affine sets. From a combinatorial
point of view $X$ is specified by a finite fan~$\fan$ in $\bZ^n
\tensor \bR \iso \bR^n$, and each cone $\sigma \in \fan$ corresponds
to an $A$-algebra $A^\sigma$ and hence to an open affine set $U_\sigma
= \mathrm{Spec} (A^\sigma) \subseteq X$. By evaluating on the open
sets~$U_\sigma$ we see that a chain complex~$Y$ of
quasi-coherent sheaves on~$X_\fan$ can thus be specified by a
collection of $A^\sigma$-module chain complexes $Y^\sigma$ for $\sigma
\in \fan$, subject to certain compatibility conditions. These include,
among other things, isomorphisms of chain complexes
\begin{equation}
  \label{eq:compat_iso}
  A^\tau \tensor_{A^\sigma} Y^\sigma \iso Y^\tau
\end{equation}
for all pairs of cones $\tau \subseteq \sigma$ in~$\fan$; in the
language of sheaves, this means that we recover $Y^\tau$ by
restricting the sections $Y^\sigma$ over~$U_\sigma$ to the smaller
open set~$U_\tau$.

The main result of this paper is that the derived category of~$X_\fan$
can be described using collections of chain complexes which do not
necessarily satisfy the compatibility
condition~(\ref{eq:compat_iso}). In more technical parlance, we will
prove that the category of (twisted) diagrams
\[\fan^\op \rTo \hbox{chain complexes}, \ \sigma \mapsto Y^\sigma\]
admits a ``colocal'' model structure whose homotopy category is
equivalent to the (unbounded) derived category $D(\qco(X_\fan))$,
cf.~Theorem~\ref{thm:D_of_X_from_diagrams}. In the process we will
also identify explicitly a finite set of weak generators
of~$D(\qco(X_\fan))$, cf.~Construction~\ref{constr:R_fan}. In case
$\fan$ is a complete fan, the description is particularly simple: It
suffices to take one line bundle $\mathcal{O(\vec\sigma)}$ for each
cone $\sigma \in \fan$, cf.~Example~\ref{example:complete_fan} and
Corollary~\ref{cor:weak_generators_of_D}.

The cofibrant objects of the colocal model structure are characterised
by a weak form of compatibility condition
(Theorem~\ref{thm:homsheaf_colocal}): Instead of requiring
isomorphisms as in~(\ref{eq:compat_iso}) we ask for quasi-isomorphisms
\[A^\tau \tensor_{A^\sigma} Y^\sigma \simeq Y^\tau\] for all pairs of
cones $\tau \subseteq \sigma$ in~$\fan$. We call the resulting
structure a {\it homotopy sheaf}. Clearly every chain complex of
quasi-coherent sheaves is a homotopy sheaf.

A main ingredient of the proof is that the homotopy category of
homotopy sheaves is nothing but the (unbounded) derived category of
quasi-coherent sheaves on~$X_\fan$
(Theorem~\ref{thm:derived_is_correct}); this result is valid for
arbitrary toric schemes defined over a commutative ring~$A$, and
holds more generally for quasi-compact $A$-schemes equipped with a
finite semi-separating affine covering. Note that every quasi-compact
separated scheme can be equipped with such a covering. The main
technical result is that homotopy sheaves can be replaced, up to
quasi-isomorphism on the covering sets, by quasi-coherent sheaves
(Lemma~\ref{lem:pullback_gluing}).

The paper illustrates the philosophy that homotopy sheaves are a
flexible substitute for quasi-coherent sheaves which allow for easier
handling in a homotopy-theoretic setting.

\medbreak

We will use the language of \textsc{Quillen} model categories as
presented by \textsc{Dwyer} and \textsc{Spalinski}
\cite{Dwyer-Spalinski}, \textsc{Hirschhorn} \cite{Hirsch-Loc} and
\textsc{Hovey} \cite{MR1650134}. Another essential ingredient is the
language of toric varieties, and the corresponding combinatorial
objects (cones and fans); a full treatment can be found in
\textsc{Fulton}'s book~\cite{Fulton-toric}. We will also have occasion
to use variants of diagram categories and their associated model
category structures as introduced by \textsc{R\"ondigs} and the
author~\cite{HR-Twisted}.

\section{Chain complexes}

\subsection{Model structure and resolutions}

Let $A$ denote a ring with unit.  The category $\mathrm{Ch}_A$ of
(possibly unbounded) chain complexes of left $A$-modules will be
considered with the {\it projective model structure\/}: Weak
equivalences are the quasi-isomorphisms, and fibrations are those maps
which are surjective in each degree
\cite[Theorem~2.3.11]{MR1650134}. A particularly convenient feature of
this model structure is that {\it all chain complexes are fibrant}.

Also of interest is the full subcategory $\mathrm{Ch}_A^+$ of
non-negative chain complexes. It is a model category with weak
equivalences and cofibrations as before, but with fibrations the maps
which are surjective in positive degrees
\cite[Theorem~7.2]{Dwyer-Spalinski}. The category $\mathrm{Ch}_A^+$ is
equivalent to the category $\mathrm{sMod}_A$ of simplicial
$A$-modules; the equivalence is given by the reduced chain complex
functor $N \colon \mathrm{sMod}_A \rTo \mathrm{Ch}_A^+$ and its
inverse, the \textsc{Dold\/}-\textsc{Kan} functor~$W$. Given a chain
complex $C \in \mathrm{Ch}_A^+$ the result of applying~$W$ is the
simplicial $A$-module
\[\bN \ni n \mapsto \hom_{\mathrm{Ch}_A} (N(A[\Delta^n]),\, C)\]
where $\Delta^n$ denotes the standard $n$-simplex. The functors~$N$
and~$W$ preserve and detect weak equivalences.

Note that we can consider $N$ as a functor with values in the category
$\mathrm{Ch}_A$. Similarly, the definition of~$W$ above makes sense even
if~$C$ is an unbounded chain complex. In this context, the following
is known to be true:

\begin{lemma}
  \label{lem:dold_kan_adjunction}
  Let $N \colon \mathrm{sMod}_A \rTo \mathrm{Ch}_A$ and $W \colon
  \mathrm{Ch}_A \rTo \mathrm{sMod}_A$ be defined as above.
  \begin{enumerate}
  \item[{\rm (1)}] The functor $N$ is left \textsc{Quillen\/} with right
    adjoint~$W$.
  \item[{\rm (2)}] The functor $N$ preserves and detects weak equivalences.
  \item[{\rm (3)}] A map~$f$ of chain complexes induces an $H_n$-isomorphism for
    all $n \geq 0$ if and only if $W(f)$ is a weak equivalence of
    simplicial modules. \qed
  \end{enumerate}
\end{lemma}

\begin{lemma}
  \label{lem:ch_A_is_cellular}
  The category $\mathrm{Ch}_A$ is a cellular model category in the
  sense of \cite[\S12]{Hirsch-Loc}; the set of generating cofibrations
  is
  \[I:= \{ S_{n-1}(A) \rTo D_n(A) \,|\, n \in \bZ \} \ ,\]
  and the set of generating acyclic cofibrations is
  \[J:= \{ 0 \rTo D_n(A) \,|\, n \in \bZ \} \ .\]
  Here $S_k(A)$ denotes the chain complex which has $A$ in degree~$k$
  and is trivial everywhere else, and $D_n(A)$ denotes the chain
  complex which has $A$ in degrees $n$ and $n-1$ with boundary map the
  identity, and is trivial everywhere else.
\end{lemma}

\begin{proof}
  This is the content of~\cite[Theorem~2.3.11]{MR1650134}. 
\end{proof}

\begin{lemma}
  \label{lem:cosimiplicial_resolution_ch_a}
  Let $C \in \mathrm{Ch}_A$ be a cofibrant chain complex. The
  cosimplicial chain complex $N(A[\Delta^\bullet]) \tensor_A C$, \ie,
  the cosimplicial object
  \[\bN \ni n \mapsto N(A[\Delta^n]) \tensor_A C \ ,\]
  defines a cosimplicial resolution \cite[\S16.1] {Hirsch-Loc} of~$C$; the
  structure map to the constant cosimplicial object $\mathrm{cc}^* C$ is
  induced by the unique map $\Delta^n \rTo \Delta^0$ and the natural
  isomorphism $N(A[\Delta^0]) \tensor_A C \iso C$. The $n$-th latching object
  is the chain complex $L_n (N(A[\Delta^\bullet]) \tensor_A C) =
  N(A[\partial \Delta^n]) \tensor_A C$.
\end{lemma}

\begin{proof}
  The category of cosimplicial objects in~$\mathrm{Ch}_A$ carried a
  \textsc{Reedy} model structures \cite[\S15.3]{Hirsch-Loc}.  To prove
  the Lemma, the non-trivial thing to verify is that
  $N(A[\Delta^\bullet]) \tensor_A C$ is cofibrant with respect to this
  model structure.

  The category of cosimplicial simplicial $A$-modules
  carries a \textsc{Reedy} model structure as well. The object
  $A[\Delta^\bullet]$ is known to be cofibrant, so for all $n \in \bN$
  the latching map \cite[Proposition~16.3.8~(1)]{Hirsch-Loc}
  \[A[\partial \Delta^n] = A[\Delta^\bullet] \tensor \partial \Delta^n
  = L_n A[\Delta^\bullet] \rTo A[\Delta^n] = A[\Delta^\bullet] \tensor
  \Delta^n\]
  is a cofibration of simplicial $A$-modules. Hence we have a
  cofibration of chain complexes
  \[N(L_n A[\Delta^\bullet]) \rTo N(A[\Delta^n])\]
  since the functor~$N$ is left \textsc{Quillen} by
  Lemma~\ref{lem:dold_kan_adjunction}.  Now the functor $N$, being a
  left adjoint, commutes with colimits so that the source of this map
  is isomorphic to $L_n N(A[\Delta^\bullet])$. Taking tensor product
  with a cofibrant chain complex preserves cofibrations and commutes
  with colimits, so by applying $\,\cdot\, \tensor_A C$ we see that
  the latching map
  \[L_n (N(A[\Delta^\bullet]) \tensor_A C) \iso L_n N
  (A[\Delta^\bullet]) \tensor_A C \rTo N(A[\Delta^n]) \tensor_A C\]
  of $N(A[\Delta^\bullet]) \tensor_A C$ is a cofibration as required.
\end{proof}

\subsection{Homotopy limits of diagrams of chain complexes}
\label{sec:homoty-limit}

\begin{definition}
  \label{def:path_factorisation}
  Let $f \colon C \rTo D$ be a map of (possibly) unbounded chain
  complexes. The {\it canonical path space factorisation of~$f$\/} is the
  factorisation $C \rTo^i P(f) \rTo^p D$ where the degree~$n$ part of $P(f)$ is
  $C_n \times D_{n+1} \times D_n$ with differential as specified in the
  following diagram:
  \begin{diagram}[labelstyle=\scriptstyle]
    C_n & \times & D_{n+1} & \times & D_n \\
    \dTo<\partial & \rdTo<{-f} & \dTo<{-\partial} & \ldTo>= & \dTo>\partial\\
    C_{n-1} & \times & D_n & \times & D_{n-1}
  \end{diagram}
  The map $i = (\id, 0, f)$ is a chain homotopy equivalence (with homotopy
  inverse given by $\mathrm{pr}_1$). The map $p = \mathrm{pr}_3$ is levelwise
  surjective, hence $p$~is a fibration in~$\mathrm{Ch}_A$ (in the
  projective model structure).
\end{definition}

In what follows, we will be concerned with diagrams indexed by a
finite fan~$\fan$. A {\it cone\/} in a finite-dimensional real vector
space $N_\bR$ is the positive span of a finite set of vectors
of~$N_\bR$. A {\it fan\/} is a finite collection of cones $\fan =
\{\sigma_1,\, \sigma_2,\, \ldots,\, \sigma_k\}$ which is closed under
taking faces, and satisfies the condition that the intersection of two
cones in~$\fan$ is a face of both cones. We also require that all the
cones are {\it pointed\/}, \ie, have the trivial cone~$\{0\}$ as a
face.  We consider a fan~$\fan$ as a poset ordered by inclusion of
cones or, equivalently, as a category with morphisms given by
inclusion of cones. The trivial cone~$\{0\}$ is initial in the
category~$\fan$.---By abuse of language, we refer to $\dim (N_\bR)$ as
the dimension of~$\fan$.

\begin{definition}
  \label{canonical-f-replacement}
  Let $\fan$ denote a finite fan. Given a diagram of chain complexes
  \[C \colon \fan^\op \rTo \mathrm{Ch}_A, \quad \sigma \mapsto C^\sigma\]
  we define its {\it canonical fibrant replacement}
  \[PC \colon \fan^\op \rTo \mathrm{Ch}_A\]
  inductively as follows. To begin with, set $(PC)^{\{0\}} =
  C^{\{0\}}$. For every $1$-dimensional cone $\rho \in \fan$ factor
  the map $f \colon C^\rho \rTo (PC)^{\{0\}}= C^{\{0\}}$ as
  \[C^\rho \rTo P(f) \rTo (PC)^{\{0\}} \ ,\]
  see Definition~\ref{def:path_factorisation}, and set $(PC)^\rho =
  P(f)$.  Now continue by induction on the dimension: Given a
  positive-dimensional cone $\sigma \in \fan$, factor the map $f
  \colon C^\sigma \rTo \lim_{\tau \subset \sigma} (PC)^\tau$ as
  \[C^\sigma \rTo P(f) \rTo \lim_{\tau \subset \sigma} (PC)^\tau \ ,\]
  and define $(PC)^\sigma = P(f)$.
\end{definition}

There resulting map of diagrams $C \rTo PC$ is an objectwise injective
weak equivalence. By construction the diagram
$PC$ is fibrant in the sense that for all cones $\sigma \in \fan$, the
map
\[(PC)^\sigma \rTo \lim_{\tau \subset \sigma} (PC)^\tau\]
is surjective (the limit taken over all cones strictly contained
in~$\sigma$). The terminology relates to a model structure on the
category of $\fan^\op$-diagrams in $\mathrm{Ch}_A$ with objectwise
weak equivalences and cofibrations.

The passage from $C$ to~$PC$ is functorial in~$C$ and maps objectwise
weak equivalences to objectwise weak equivalences.

\begin{definition}
  \label{def:homotopy-limit}
  Let $\fan$ denote a finite fan as before, and let $C$ denote a
  diagram of chain complexes
  \[C \colon \fan^\op \rTo \mathrm{Ch}_A, \quad \sigma \mapsto
  C^\sigma \ .\]
  The {\it homotopy limit\/} $\holim\,(C) = \holim_{\fan^\op} (C)$
  of~$C$ is defined as
  \[\holim (C) := \lim PC \ .\]
  The homology modules of~$\holim (C)$ are called the {\it
    hyper-derived inverse limits\/} of the diagram~$C$.
\end{definition}

\begin{remark}
  \label{remark:holim}
  \begin{enumerate}
  \item[{\rm (1)}] If $\fan$ has a unique (inclusion-)maximal cone~$\mu$, then
\[\holim (C) = \lim PC \iso (PC)^\mu \ ,\]
    so $C^\mu \simeq \holim (C)$ induced by the quasi-isomorphism $C^\mu
    \rTo^\simeq (PC)^\mu$.
  \item[{\rm (2)}] If $D$ is a $\fan$-indexed diagram of $A$-modules, viewed as a
    diagram of chain complexes concentrated in degree~$0$, then the
    homotopy limit computes higher derived inverse limit:
    \[h_{-k} \holim (D) \iso \lim{}\!^k (D) \ .\]
    Of course $\lim{}\!^k (D)$ will be trivial in this case unless $0
    \leq k \leq n$.
  \end{enumerate}
\end{remark}

The homotopy limit construction is invariant under weak equivalences
of diagrams. That is, if $f \colon C \rTo D$ is an objectwise
quasi-isomorphism then the induced map $\holim (C) \rTo \holim (D)$ is
a quasi-isomorphism.

\begin{lemma}
  \label{lem:holim_constant}
  Let $C$ be a chain complex of $A$-modules, and let
  $\mathrm{con}(C)$ denote the constant $\fan^\op$-diagram with
  value~$C$. Then $C \simeq \holim (\mathrm{con}(C))$.
\end{lemma}

\begin{proof}
  Since $\fan^\op$ has terminal object $\{0\}$, it is easy to see that
  for $\sigma \neq \{0\}$ the map
  \[C = \mathrm{con}(C)^\sigma \rTo \lim_{\tau \subset \sigma}
  \mathrm{con}(C)^\tau = C\]
  is the identity. This means that $\mathrm{con}(C)$ is fibrant in the
  model structure mentioned above. Hence the canonical map
  $\mathrm{con}(C) \rTo P\mathrm{con}(C)$ is a weak equivalence of
  fibrant diagrams. Consequently, the right \textsc{Quillen} functor
  ``inverse limit'' yields a quasi-isomorphism
  \[C = \lim \mathrm{con}(C) \rTo^\sim \lim P \mathrm{con}(C) = \holim
  (\mathrm{con}(C))\]
  by application of \textsc{Brown}'s Lemma \cite[dual of
  Lemma~9.9]{Dwyer-Spalinski}.
\end{proof}

\section{Presheaves and line bundles on toric schemes}

\subsection{Toric schemes}
\label{subsec:toric-varieties}

Let $N \iso \bZ^n$ denote a lattice of rank~$n$. Write $N_\bR = N
\tensor \bR \iso \bR^n$. There is an obvious inclusion $N \subseteq
N_\bR$ given by identifying $p \in N$ with $p \otimes 1 \in N_\bR$.
We denote the dual lattice of~$N$ by the letter~$M$, and write $M_\bR
= M \tensor \bR$. Clearly $M \subseteq M_\bR$, and $M_\bR$ is the dual
vector space of~$N_\bR$.

Let $\fan$ be a finite fan in~$N_\bR$,
cf.~\S\ref{sec:homoty-limit}. In addition to the conditions listed
there, we require each cone in~$\fan$ to be {\it rational}, \ie,
spanned by finitely many vectors in~$N \subset N_\bR$ We write
$\fan(1)$ for the set of $1$-cones in~$\fan$.  Similarly, if $\sigma
\in \fan$ is any cone we write $\sigma(1)$ for the set of $1$-cones
of~$\fan$ contained in~$\sigma$. Every $1$\nobreakdash-cone~$\rho$ is
spanned by a unique primitive element $n_\rho \in N$; the set
$\{n_\rho \,|\, \rho \in \sigma(1)\}$ is called the set of primitive
generators of $\sigma \in \fan$.

A cone $\sigma \in \fan$ then gives rise to a pointed monoid
\begin{equation}
  \label{eq:S_sigma}
  S_\sigma = \{ f \in M \,|\, \forall \rho \in \sigma(1): f(n_\rho)
  \geq 0 \}_+
\end{equation}
where the subscript ``+'' means adding a new element $*$ which acts
like $a + * = * + a = *$ for all $a \in S_\sigma$; this convention
will be useful when describing restriction functors
in~\S\ref{sec:restriction}.  The cone~$\sigma$ thus determines an
$A$-algebra
\[A^\sigma = \tilde A[S_\sigma]\]
where $A$ is any ring with unit (possibly non-commutative), and
$\tilde A[S_\sigma]$ is the reduced monoid algebra $A[S_\sigma]/A[*]$
of~$S_\sigma$.

\medbreak

In case $A$~is a commutative ring, we set $U_\sigma =
\mathrm{Spec}(A^\sigma)$, and define the $A$-scheme $X_\fan$ as the
union $\bigcup_{\sigma \in \fan} U_\sigma$. By construction, $U_\sigma
\cap U_\tau = U_{\sigma \cap \tau}$ for all cones $\sigma, \tau \in
\fan$. The scheme $X_\fan$ is called {\it the toric scheme associated
  to~$\fan$}. If $A$ is an algebraically closed field, $X_\fan$ is an
algebraic variety over~$A$. See \textsc{Fulton} \cite{Fulton-toric}
for a full treatment of toric varieties, and more details of the
construction.

\subsection{Presheaves on toric schemes}

As before let $\fan$ denote a finite fan of rational pointed cones,
and let $A$ denote a (possibly non-commutative) ring with unit. For
commutative $A$ this data defines an $A$-scheme $X_\fan$ as indicated
in \S\ref{subsec:toric-varieties}.  But even if $A$ is non-commutative
we will speak of presheaves on~$X_\fan$:

\begin{definition}
  \label{def:preheaves}
  The category $\pre(\fan)$ of {\it presheaves on the toric scheme $X_\fan$
    defined over~$A$\/} has objects the diagrams
  \[C \colon \fan^\op \rTo \mathrm{Ch}_A, \quad \sigma \mapsto C^\sigma\]
  together with additional data that equip each entry $C^\sigma$ with the
  structure of an object of $\mathrm{Ch}_{A^\sigma}$, and such that for each
  inclusion $\tau \subseteq \sigma$ in $\fan$ the structure map $C^\sigma \rTo
  C^\tau$ is $A^\sigma$-linear.
%   We call $C$ {\it bounded\/} (\resp, {\it
%     bounded below\/}) if all its constituents are bounded (\resp, bounded
%   below) chain complexes.
\end{definition}

A particularly useful example of a presheaf is the functor
\[\mathcal{O} = \mathcal{O} (\vec 0) \colon \fan^\mathrm{op} \rTo
\mathrm{Ch}_A, \quad \sigma \mapsto A^\sigma\]
(see~\S\ref{sec:comb-proj}) where we consider the algebra $A^\sigma$ as an
$A^\sigma$-module chain complex concentrated in degree~$0$.

\subsection{Model structures}

The category $\pre(\fan)$ defined above is an example of a twisted
diagram category in the sense of~\cite[\S2.2]{HR-Twisted}, formed with
respect to an adjunction bundle similar to the one described in
Example~2.5.4 of {\it loc.cit.\/} (one needs to replace ``modules''
with ``chain complexes of modules''). We thus know that the category
$\pre(\fan)$ has two \textsc{Quillen\/} model structures, called the
$f$-structure and the $c$-structure, respectively. In both cases the
weak equivalences are the objectwise quasi-isomorphisms. Fibrations
and cofibrations are different, as explained below.

\begin{numpar}
  \label{sec:f-structure}
  {\bf The $f$-structure} \cite[Theorem~3.3.5]{HR-Twisted}. In this
  model structure, a map $f \colon C \rTo D$ in $\pre(\fan)$ is a
  cofibration if and only if all its components $f^\sigma$, $\sigma
  \in \fan$, are cofibrations in their respective categories.

  Fibrations can be characterised using {\it matching complexes}. For
  $\sigma \in \fan$ define $M^\sigma (C):= \lim_{\tau \subset \sigma}
  C^\tau$, the limit taken in the category $\mathrm{Ch}_{A^\sigma}$
  over all $\tau \in \fan$ properly contained in~$\sigma$. Then $f
  \colon C \rTo D$ is a fibration if and only if for all $\sigma \in
  \fan$ the induced map $\iota \colon C^\sigma \rTo M^\sigma (C)
  \times_{M^\sigma (D)} D^\sigma$ is a fibration
  in~$\mathrm{Ch}_{A^\sigma}$ (\ie, if $\iota$ is levelwise
  surjective).
\end{numpar}

\begin{lemma}
  Let $C$ be an object of $\pre(\fan)$. The canonical fibrant
  replacement $PC$ of~$C$ as defined in~\ref{canonical-f-replacement}
  yields an $f$-fibrant object of~$\pre(\fan)$.
\end{lemma}

\begin{proof}
  The important thing to note is that for each inclusion of cones
  $\tau \subseteq \sigma$ there is an inclusion of algebras $A^\sigma
  \subseteq A^\tau$, so $C^\tau$ can be considered as an
  $A^\sigma$-module chain complex by restriction of scalars. It is
  then a matter of tracing the definitions to see that $PC \in
  \pre(\fan)$. Since fibrations are surjections in all relevant
  categories of chain complexes, and since surjectivity can be
  detected after restricting scalars to the ground ring~$A$, the Lemma
  follows.
\end{proof}

\begin{numpar}
  \label{sec:c-structure}
  {\bf The $c$-structure} \cite[Theorem~3.2.13]{HR-Twisted}. In this
  model structure, a map $f \colon C \rTo D$ in $\pre(\fan)$ is a
  fibration if and only if all its components $f^\sigma$, $\sigma \in
  \fan$, are fibrations in their respective categories (\ie, the
  components are surjective in all chain levels). Note that all
  objects of $\pre(\fan)$ are $c$-fibrant.
  
  Cofibrations can be characterised using {\it latching complexes}.
  For $\sigma \in \fan$ define $L_\sigma (C):= \mathrm{colim}_{\tau
    \supset \sigma}\, A^\sigma \tensor_{A^\tau} C^\tau$, the colimit
  being taken over all $\tau \in \fan$ properly containing~$\sigma$.
  Then $f \colon C \rTo D$ is a cofibration if and only if for all
  $\sigma \in \fan$ the map
  \[L_\sigma (D) \cup_{L_\sigma (C)} C^\sigma \rTo D^\sigma\]
  is a cofibration in $\mathrm{Ch}_{A^\sigma}$. In particular, $D$ is
  cofibrant if and only if for all $\sigma \in \fan$ the map $L_\sigma
  (D) \rTo D^\sigma$ is a cofibration.

  For $\tau \in \fan$ and $P \in \mathrm{Ch}_A$ we define the diagram
  \[F_\tau (P) \colon \sigma \mapsto
  \begin{cases}
    0 & \hbox{if\ }  \sigma \not\subseteq \tau \\
    A^\sigma \tensor_A P & \hbox{if\ } \sigma \subseteq \tau
  \end{cases}
  \]
  together with the evident structure maps induced by the various
  inclusions of $A$-algebras $A^\sigma \rTo A^{\sigma^\prime}$.
\end{numpar}

\begin{lemma}
  \label{lem:c_structure_is_cellular}
  The $c$-structure is a cellular model structure in the sense of
  \cite[\S12.1]{Hirsch-Loc}. A set of generating cofibrations is given
  by
  \[I_c := \left\{ F_\tau (i) \,|\, i \in I,\ \tau \in \fan \right\}\]
  where $I$ is as in Lemma~\ref{lem:ch_A_is_cellular}. Similarly, a
  set of generating acyclic cofibrations is
  \[J_c := \left\{ F_\tau (j) \,|\, j \in J,\ \tau \in \fan \right\}\]
  with $J$ as in Lemma~\ref{lem:ch_A_is_cellular}.
\end{lemma}

\begin{proof}
  This follows by direct inspection from
  Lemma~\ref{lem:ch_A_is_cellular}. We omit the details.
\end{proof}

\begin{lemma}
  \label{lem:cosimplicial_resolution_pre_fan}
  Suppose $C \in \pre(\fan)$ is a $c$-cofibrant object
  (\ref{sec:c-structure}). Then
  \[A[\Delta^\bullet] \tensor C \colon \ \fan^\mathrm{op} \rTo
  \mathrm{Ch}_A, \quad \sigma \mapsto A[\Delta^\bullet] \tensor_A
  C^\sigma \]
  is a cosimplicial resolution of~$C$.
\end{lemma}

\begin{proof}
  This follows from the fact that $A[\Delta^\bullet]$ is
  \textsc{Reedy} cofibrant cosimplicial simplicial module, and the
  fact the taking tensor products commutes with colimits. The details
  are similar to Lemma~\ref{lem:cosimiplicial_resolution_ch_a}.
\end{proof}

\subsection{Restriction and extension by zero}
\label{sec:restriction}

We will use the notation of \S\ref{subsec:toric-varieties}.  Let
$\fan$ denote a finite fan in~$N_\bR$. Given a cone $\rho \in \fan$ we
define the {\it star of~$\rho$\/} as
\[\st(\rho) = \{ \sigma \in \fan \,|\, \rho \subseteq \sigma \} \ .\]

\begin{numpar}
  \label{par:quotient_fan}
  A $1$-cone $\rho \in \fan(1)$ determines a fan $\fan/\rho$ in an
  $(n-1)$-dimensional vector space as follows. Let $\bZ\rho$ denote
  the sub-lattice of~$N$ generated by the span of~$\rho$. Then $\bar N
  = N/\bZ\rho$ is a lattice of rank~$n-1$.  Given any cone $\sigma \in
  \st(\rho)$ the image $\bar \sigma$ of~$\sigma$ under the projection
  $N_\bR \rTo \bar N_\bR$ is a pointed rational polyhedral cone, and
  by varying~$\sigma \in \st(\rho)$ we obtain a fan $\fan/\rho$ of a
  toric scheme denoted $X_{\fan/\rho} = V_\rho$.  Note that this new
  fan is isomorphic, as a graded poset, to $\st(\rho)$.---If $A = \bC$
  then $V_\rho$ is the closure of the orbit in~$X_\fan$ corresponding
  to~$\rho$, and its is known that $V_\rho$ has codimension~$1$ in
  $X_\fan$.
\end{numpar}

From now on we will assume
% that the fan $\fan$ is {\it complete}, that
% is, that $\bigcup_{\sigma \in \fan} \sigma = N_\bR$, and
that the fan is {\it regular}, that is, each cone of $\fan$ is spanned
by part of a $\bZ$-basis (which depends on the cone under
consideration) of the lattice~$N$.  This condition is equivalent to
the requirement that the toric variety $X_\fan$ defined over~$\bC$ is
smooth.

\medbreak

Given $\rho \in \fan(1)$ and $\sigma \in \st(\rho)$ let $n_1, \ldots,
n_k$ denote the primitive elements of the $1$-cones contained
in~$\sigma$.  Suppose that $n_k \in \rho$ (which can be achieved by
renumbering).  Let $\bar\sigma$ denote the image of~$\sigma$ in $\bar
N_\bR = (N/\bZ\rho)_\bR$ as before, and denote the images of the~$n_j$
in~$\bar N$ by $\bar n_j$. Since $\sigma$ is regular, the $\bar n_1,
\ldots, \bar n_{k-1}$ form part of a basis of the lattice~$\bar N$,
and are precisely the primitive elements of the $1$-cones contained
in~$\bar\sigma$. Since the lattice dual of~$\bar N$ is $M \cap
\rho^\perp$, we see that
\[S_{\bar\sigma} \iso \{ f \in M \,|\, f(n_j) \geq 0 \hbox{ for } 1 \leq
j \leq k-1, \hbox{ and } f|_\rho =0 \}_+ \]
(compare to the description~(\ref{eq:S_sigma}) of the monoid~$S_\sigma$).
Of course $f|_\rho = 0$ is equivalent to $f(n_k) = 0$.---We obtain a
surjective map of pointed monoids
\begin{equation}
  \label{eq:surj_Ssigma_Sbarsigma}
  S_\sigma \rTo S_{\bar\sigma}, \quad f \mapsto
  \begin{cases}
    f & \hbox{if\ } f|_\rho =0 \\ * & \hbox{else}
  \end{cases}
\end{equation}
and, by linearisation, a corresponding surjective map of $A$-algebras
\begin{equation}
  \label{eq:surj_Asigma_Abarsigma}
  A^\sigma \rTo A^{\bar\sigma} \ .
\end{equation}
For commutative~$A$ this map exhibits $\mathrm{Spec}(A^{\bar \sigma})
= V_\rho \cap U_\sigma$ as a closed subset of $U_\sigma \subseteq
X_\fan$.

\begin{numpar}
  Recall that the fan $\fan/\rho$ of~$V_\rho$ is isomorphic, as a poset, to
  $\st(\rho) \subseteq \fan$. Thus an object $C \in \pre(\fan/\rho)$ can be
  considered as a functor defined on the poset $\st(\rho)^\op$, and we define
  a diagram $\zeta(C)$ on~$\fan^\op$ by setting
  \[\zeta(C)^\sigma :=
  \begin{cases}
    0 & \hbox{if\ } \rho \not\subseteq \sigma \\
    C^\sigma & \hbox{if\ } \rho \subseteq \sigma
  \end{cases}
  \] 
  with structure maps induced by those of~$C$. For $\sigma \in \st(\rho)$ we
  let $A^\sigma$ act on $\zeta(C)^\sigma$ via the surjection $A^\sigma \rTo
  A^{\bar\sigma}$. In this way, $\zeta (C)$ becomes an object of $\pre(\fan)$,
  called the {\it extension by zero\/} of~$C$. By direct computation we
  verify:
\end{numpar}

\begin{lemma}
  \label{lem:holim_of_extension}
  For $C \in \pre(\fan/\tau)$ there is an equality
  \[\holim_{\fan^\mathrm{op}}\, C = \holim_{\st(\rho)^\mathrm{op}}\,
  \zeta(C)\]
  where we consider the presheaves on left and right hand side as
  diagrams with values in the category of $A$-modules to form the
  homotopy limits (\ref{def:homotopy-limit}). \qed
\end{lemma}

\begin{numpar}
  \label{par:restriction}
  The extension functor $\zeta \colon \pre(\fan/\rho) \rTo \pre(\fan)$
  has a left adjoint~$\varepsilon$, called {\it restriction
    to~$V_\rho$}. Its effect on $C \in \pre(\fan)$ is the following:
  As a diagram of $A$-module chain complexes, $\varepsilon (C)$ is
  given by
  \[\varepsilon(C) \colon \st(\rho)^\op \rTo \mathrm{Ch}_A, \quad M
  \mapsto A^{\bar\sigma} \tensor_{A^\sigma} C^\sigma \ ,\]
  the tensor product formed with respect to the surjection $A^\sigma
  \rTo A^{\bar\sigma}$. We also denote $\varepsilon (C)$ by
  $C|_{V_\rho}$.
\end{numpar}

\subsection{Line bundles and twisting}
\label{sec:comb-proj}

As before, let $\fan$ denote a regular fan in~$N_\bR$, and recall that
every $1$-cone $\rho \in \fan$ is generated by a unique primitive
element $n_\rho \in N$.

\begin{construction}
  \label{def:twisting_sheaf}
  Fix a vector $\vec k = (k_\rho)_{\rho \in \fan(1)} \in
  \bZ^{\fan(1)}$.  Since $\fan$ is regular we can find for every cone
  $\sigma \in \fan$ an integral linear form $f_\sigma \colon N_\bR
  \rTo \bR$, unique up to adding a linear form vanishing on~$\sigma$,
  which satisfies $f_\sigma(n_\rho) = -k_\rho$ for every $1$-cone
  $\rho$ contained in~$\sigma$.
  
  If $\tau \in \fan$ is another cone, then $f_\tau$ and $f_\sigma$
  agree on~$\tau \cap \sigma$ (since they agree on $1$-cones of $\tau
  \cap \sigma$), and both $\pm (f_\tau - f_\sigma)$ are elements
  of~$S_{\tau \cap \sigma}$. Consequently we have $f_\tau + S_{\tau
    \cap \sigma} = f_\sigma + S_{\tau \cap \sigma}$; in particular,
  the set $f_\sigma + S_\sigma$ depends on~$\sigma$ and~$\vec k$ only
  (and not the specific choice of function~$f_\sigma$). We thus
  obtain a well-defined functor
  \[\mathcal{O}(\vec k) \colon \fan^\op \rTo A\hbox{-}\mathrm{mod}, \quad \tau
  \mapsto \tilde A[f_\tau + S_\tau] \ ,\] considered as a diagram of
  chain complexes concentrated in degree~$0$. Structure maps are given
  by inclusions. We call $\mathcal{O}(\vec k)$ the {\it line bundle
    determined by~$\vec k$}. Note that $\mathcal{O}(\vec k)$ is, in
  fact, an object of~$\pre(\fan)$ (as usual, we think of modules as
  chain complexes concentrated in degree~$0$): The action of $S_\tau$
  on~$f_\tau + S_\tau$ extends to an $A^\tau$-module structure of
  $\tilde A[f_\tau + S_\tau]$, and for $\rho \subseteq \tau$ the
  structure maps $\mathcal{O}(\vec k)^\tau \rTo \mathcal{O}(\vec
  k)^\rho$ are easily seen to be linear with respect to the
  ring~$A^\tau$.
\end{construction}

In effect the vector $\vec k \in \bZ^{\fan(1)}$, or rather the
collection of the~$f_\sigma$, determines a piecewise linear function
on the underlying space of~$\fan$, and we have given a combinatorial
description of the associated line bundle on~$X_\fan$.

\begin{example}
  Let $\fan$ denote the fan of the projective line; it is a fan
  in~$\bR$ with $1$-cones the non-positive and non-negative real
  numbers, respectively. For a vector $\vec k = (k_1, k_2) \in \bZ^2$
  the diagram $\mathcal{O}(\vec k)$ then has the form
  \[T^{k_1} \cdot A[T^{-1}] \rTo^\subset A[T,T^{-1}] \lTo^\supset
  T^{-k_2} \cdot A[T]\]
  which, as a quasi-coherent sheaf, is isomorphic to the algebraic
  geometers' sheaf $\mathcal{O}_{\mathbb{P}^1} (k_1+k_2)$.
\end{example}

In general, recall that $S_\tau = \{g \in M \,|\,
\forall \rho \in \tau(1) \colon g(n_\rho) \geq 0\}_+$. The map $g
\mapsto f_\tau + g$ defines an $S_\tau$-equivariant bijection
from~$S_\tau$ to
\begin{equation}
  \label{eq:def_B_k_tau}
  B(\vec k)_\tau := f_\tau + S_\tau = \{ g \in M \,|\, \forall \rho
  \in \tau(1) \colon g(n_\rho) \geq -k_\rho \}_+ \ .
\end{equation}
In particular, {\it $\mathcal{O}(\vec k)^\tau$ is a free
  $A^\tau$-module of rank~$1$}.
  
From the construction it is clear that {\it given another vector $\vec
  \ell \in \bZ^{\fan(1)}$ with $\vec\ell \leq \vec k$ (componentwise
  inequality) we have a canonical injection (inclusion map)
  $\mathcal{O}(\vec\ell) \rTo \mathcal{O}(\vec k)$}.
  % The most
% important ones for us are of the form $\mathcal{O}(\vec k) \rTo
% \mathcal{O} (\vec k + e_\rho)$ where $e_\rho \in \bZ^{\fan(1)}$ is the
% $\rho$-th unit vector. We will determine the cofibre in
% Proposition~\ref{cofibre-of-inclusion} below.

\begin{lemma}
  \label{lem:monoid_generators}
  Given a $1$-cone $\rho \in \fan(1)$ and a cone~$\sigma$ properly
  containing~$\rho$, let $\tau \in \fan$ denote the maximal face
  of~$\sigma$ not containing~$\rho$ (this is well-defined since $\fan$
  is regular). Let $f \in M$ be a linear form which takes the
  value~$1$ on the primitive generator of~$\rho$, and takes the value
  $0$ on the primitive generators of~$\tau$.  Then $f \in S_\sigma$,
  and $S_\tau = S_\sigma + \bZ f$. In other words, the monoid $S_\tau$
  is obtained from~$S_\sigma$ by inverting the element~$f$.
\end{lemma}

\begin{proof}
  Let $n_1, \ldots, n_k$ be the primitive generators of~$\tau$, and
  let $n_{k+1}$ be the primitive generator of~$\rho$.

  A liner form $g \in M$ is in $S_\sigma$ if and only if it evaluates
  to non-negative numbers on primitive generators of~$\sigma$, \ie, if
  and only if $g(n_i) \geq 0$ for $1 \leq i \leq k+1$. So $f \in
  S_\sigma$ as claimed.

  Similarly, we have $g \in S_\tau$ if and only if $g(n_i) \geq 0$ for
  $1 \leq i \leq k$. Thus we have the inclusion $S_\tau \supseteq
  S_\sigma + \bZ f$. For the reverse inclusion, let $g \in
  S_\tau$. Then $\big(g - g(n_{k+1}) \cdot f\big) (n_{k+1}) = 0$, so
  \[g = \big( g- g(n_{k+1}) \cdot f\big) + g(n_{k+1}) \cdot f\]
  is an element of $S_\sigma + \bZ f$ as claimed.
\end{proof}

% \begin{lemma}
%   \label{lem:line_bundle_independent}
%   Let $\rho \in \fan(1)$ and $\vec k \in \bZ^{\fan(1)}$ be given. For
%   a linear map $g \in M$ set $k^\prime_\tau = k_\tau - g(n_\tau)$ for
%   $\tau \in \fan(1)$.  Then the map $M \rTo M, \quad p \mapsto p - g$,
%   or rather its restriction to bases (see (\ref{eq:def_B_k_tau})
%   above)
%   \[B(\vec k^\prime)_\tau \rTo B(\vec k)_\tau, \quad p \mapsto p-g
%   \quad (\tau \in \fan)\ ,\] 
%   defines an isomorphism $\mathcal{O}(\vec k^\prime) \rTo^\iso
%   \mathcal{O}(\vec k)$ in $\pre(\fan)$. \qed
% \end{lemma}

\begin{construction}
  \label{con:restriction_of_bundle}
  Let $\vec k \in \bZ^{\fan(1)}$ and $\rho \in \fan(1)$ be
  given. Suppose that $k_\rho = 0$.  The vector $\vec k$
  defines a line bundle on $V_\rho = X_{\fan/\rho}$ corresponding to a
  vector $\vec \ell \in \bZ^{(\fan/\rho)(1)}$ described as
  follows. Since $\fan/\rho$ is isomorphic to~$\st(\rho)$ we can write
  $\vec \ell = (\ell_\sigma)$ where $\sigma$~ranges over the
  $2$-dimensional cones in~$\st(\rho)$.  For such a cone~$\sigma$ let
  $\tau$ denote the $1$-cone contained in it different from~$\rho$,
  and set $\ell_\sigma = k_\tau$.
\end{construction}

For $\rho \in \fan(1)$ recall that the fan of~$V_\rho$ is a fan in
$(N/\bZ \rho)_\bR \iso N/\bR \rho$, and that $N/\bZ\rho$ and $M \cap
\rho^\perp$ are dual to each other. Let $\vec k \in \bZ^{\fan(1)}$
with $k_\rho = 0$. Given a cone $\bar \sigma$ in the quotient
fan, corresponding to $\sigma \in \st (\rho)$, the module
$\mathcal{O}(\vec \ell)^{\bar \sigma}$ is the reduced free $A$-module
with basis
\begin{eqnarray}
  \label{basis-for-extension}
  && \Big\{f \in M \cap \rho^\perp \,|\, f(n_\tau) \geq -k_\tau \hbox{
  for } \tau \in \sigma(1) \setminus \{\rho\} \Big\}_+ \nonumber \\
  &=& \Big\{ f \in M \,|\, f(n_\rho) = 0 \hbox{ and } f(n_\tau) \geq
  -k_\tau \hbox{ for } \tau \in \sigma(1) \setminus \{\rho\} \Big\}_+
  \ .
\end{eqnarray}
Using this explicit description, it is readily verified that
$\mathcal{O}(\vec \ell)^{\bar\sigma}$ is isomorphic to $A^{\bar
  \sigma} \tensor_{A^\sigma} \mathcal{O}(\vec k)^\sigma$, where the
tensor product is formed with respect to the surjection $A^\sigma \rTo
A^{\bar\sigma}$ from~(\ref{eq:surj_Asigma_Abarsigma}). In fact, $A^{\bar
  \sigma} \tensor_{A^\sigma} \mathcal{O}(\vec k)^\sigma$ is the
reduced free $A$-module on the pointed set $S_{\bar\sigma}
\wedge_{S_\sigma} B(\vec k)_\sigma$, formed with respect to the
surjection $S_\sigma \rTo S_{\bar\sigma}$
from~(\ref{eq:surj_Ssigma_Sbarsigma}), which is isomorphic to the set
specified in~(\ref{basis-for-extension}) above.

\begin{corollary}
  For $\rho \in \fan(1)$ and $\vec k \in \bZ^{\fan(1)}$ with $k_\rho =
  0$, let $\vec \ell$ denote the vector described in
  Construction~\ref{con:restriction_of_bundle}.  Then there is an
  isomorphism $\mathcal{O}(\vec k)|_{V_\rho} \iso \mathcal{O}(\vec
  \ell)$ of objects in~$\pre(\fan/\rho)$. In words, the restriction of
  the line bundle~$\mathcal{O}(\vec k) \in \pre (\fan)$ to
  $X_{\fan/\rho} = V_\rho$ is the line bundle~$\mathcal{O}(\vec\ell)
  \in \pre(\fan/\rho)$. \qed
\end{corollary}

Note that (\ref{basis-for-extension}) also specifies an $A$-basis of
the module $\zeta \left(\mathcal{O}(\vec \ell) \right)^\sigma$ in the
extension by zero. Using (\ref{eq:surj_Ssigma_Sbarsigma}) we can give
an explicit description of the $S_\sigma$\nobreakdash-action on this
set: The element $a \in S_\sigma$ acts by addition if $a(n_\rho) = 0$,
and acts as the zero operator if $a(n_\rho) \not= 0$.

\begin{proposition}
  \label{cofibre-of-inclusion}
  Let $\rho$ be a $1$-cone in $\fan$, and let $\vec k \in
  \bZ^{\fan(1)}$ be a vector with $k_\rho = 0$. Then the cofibre of
  the inclusion map
  \[i \colon \mathcal{O}(\vec k-\vec\rho) \rTo \mathcal{O}(\vec k)\]
  is isomorphic to the extension by zero of the restriction of
  $\mathcal{O}(\vec k)$ to~$V_\rho$. Here $\vec\rho \in
  \bZ^{\fan(1)}$ is the $\rho$-th unit vector, \ie, the vector with
  $\rho$-component~$1$ and all other entries zero.
\end{proposition}

\begin{proof}
  Let $C$ denote the cofibre of~$i$, and let $E = \zeta \big(\varepsilon
  (\mathcal{O}(\vec k))\big)$ denote the extension by zero of
  the restriction.

  Let $\sigma \in \fan \setminus \st(\rho)$ so that $\rho
  \not\subseteq \sigma$. We have $E^\sigma = 0$ by definition of
  extension, and we also have $C^\sigma = 0$ since $\mathcal{O}(\vec
  k)^\sigma = \mathcal{O}(\vec k + e_\rho)^\sigma$.  So the
  $\sigma$-components of~$C$ and~$E$ coincide in this case.

  Now let $\sigma \in \st (\rho)$. We know that $C^\sigma$ is a free
  $A$-module with pointed basis given by the cofibre of the inclusion
  of pointed sets
  \[B(\vec k-\vec\rho)_\sigma \rTo B(\vec k)_\sigma \ ,\]
  cf.~(\ref{eq:def_B_k_tau}) for notation.  Cofibres of pointed sets
  can be computed by taking complements and adding a base point. It
  follows by inspection that $C^\sigma$ has a pointed $A$-basis given
  by the set described in~(\ref{basis-for-extension}) which is also a
  pointed $A$\nobreakdash-basis of~$E^\sigma$ by the discussion
  before. Hence the $\sigma$-components of $C$ and~$E$ agree in this
  case as well.

  The reader can check that the structure maps of~$C$ and~$E$
  correspond under these identifications.
\end{proof}

\begin{definition}
  Given $\vec k \in \bZ^{\fan(1)}$ and $C \in \pre(\fan)$, we define
  the {\it $\vec k$-th twist of~$C$}, denoted $C(\vec k)$, by
  \[C(\vec k)^\sigma = \mathcal{O}(\vec k)^\sigma \tensor_{A^\sigma}
  C^\sigma\]
  with structure maps induced by those of~$C$ and~$\mathcal{O}(\vec
  k)$.
\end{definition}

This definition corresponds to tensoring a quasi-coherent sheaf with
the line bundle~$\mathcal{O}(\vec k)$, expressed in the language of
diagrams.

It is easy to check that $C (\vec k) (\vec \ell) \iso C(\vec k + \vec
\ell)$. For $\sigma$-components this comes from the isomorphism
$\mathcal{O}(\vec k)^\sigma \tensor_{A^\sigma}
\mathcal{O}(\ell)^\sigma \iso \mathcal{O}(\vec k + \vec
\ell)^\sigma$. Since $C(\vec 0) \iso C$, this proves:

\begin{lemma}
  \label{lem:twist_auto_equiv} Let $\vec k \in \bZ^{\fan(1)}$. The
    twisting functor $C \mapsto C(\vec k)$ is a self-equivalence
    of~$\pre(\fan)$ with inverse $C \mapsto C(-\vec k)$. \qed
\end{lemma}

For $\sigma \in \fan$ there is an $S_\sigma$-equivariant bijection
$B(\vec k)_\sigma \rTo S_\sigma$, cf.~(\ref{eq:def_B_k_tau}). Note
that this bijection is not canonical: It may be modified by adding or
subtracting a fixed invertible element of~$S_\sigma$. By passing to
free $A$-modules, we obtain a non-canonical isomorphism
$\mathcal{O}(\vec k)^\sigma \iso A^\sigma$ and consequently a
non-canonical isomorphism $C(\vec k)^\sigma \iso C^\sigma$. This
implies that twisting preserves and detects weak equivalences of
presheaves, preserves $c$-fibrations (objectwise surjections), and
preserves $f$-cofibrations (objectwise cofibrations). From
Lemma~\ref{lem:twist_auto_equiv} we thus conclude:

\goodbreak

\begin{corollary}
  \label{cor:twist_is_Quillen}
  Let $\vec k \in \bZ^{\fan(1)}$.
  \begin{enumerate}
  \item[{\rm (1)}] The twisting functor $C \mapsto C(\vec k)$ is a left and right
    \textsc{Quillen} functor with respect to the $c$-structure; in
    particular, if $C \in \pre(\fan)$ is $c$-cofibrant so is~$C(\vec
    k)$.
  \item[{\rm (2)}] The twisting functor $C \mapsto C(\vec k)$ is a left and right
    \textsc{Quillen} functor with respect to the $f$-structure; in
    particular, if $C \in \pre(\fan)$ is $f$-fibrant so is~$C(\vec
    k)$. \qed
  \end{enumerate}
\end{corollary}

\begin{lemma}
  \label{lem:hom_lim_iso}
  For $\vec k \in \bZ^{\fan(1)}$ and $C \in \pre(\fan)$ there are
  isomorphisms
  \[\hom_{\pre(\fan)} \big(\mathcal{O}(\vec k),\, C\big) \iso
  \hom_{\pre(\fan)} \big(\mathcal{O},\, C(-\vec k) \big) \iso \lim
  C(-\vec k) \ .\]
  These isomorphisms are natural in~$C$.
\end{lemma}

\begin{proof}
  This follows from inspection, using the trivial fact that
  $\mathcal{O}^\sigma = A^\sigma$ is the free $A^\sigma$-module of
  rank~$1$.
\end{proof}

\section{Sheaves, homotopy sheaves, and colocalisation}

\subsection{Sheaves and homotopy sheaves}

\begin{definition}
  \label{def:sheaves}
  An object $C \in \pre(\fan)$ is called a {\it (strict) sheaf\/} if for all
  inclusions $\sigma \subseteq \tau$ in $\fan$ the map
  \begin{equation}
    \label{eq:adj_str_map}
    A^\sigma \tensor_{A^\tau} C^\tau \rTo C^\sigma \ ,
  \end{equation}
  adjoint to the structure map $C^\tau \rTo C^\sigma$, is an
  isomorphism. We call $C$ a {\it homotopy sheaf\/} if the
  map~(\ref{eq:adj_str_map}) is a quasi-isomorphism for all $\sigma
  \subseteq \tau$ in $\fan$.
\end{definition}

Every strict sheaf is a homotopy sheaf. Important examples of
strict sheaves are the functors~$\mathcal{O} (\vec k)$ defined in
\S\ref{sec:comb-proj}.

\begin{lemma}
  \label{lem:homotopy_invariant}
  The notion of a homotopy sheaf is homotopy invariant: Given a weak
  equivalence $C \rTo D$ in $\pre(\fan)$, the presheaf $C$ is a
  homotopy sheaf if and only if $D$ is a homotopy sheaf.
\end{lemma}

\begin{proof}
  For all $\sigma \subseteq \tau$ in~$\fan$ the monoid $S_\sigma$ is
  obtained from~$S_\tau$ by inverting an element of~$S_\tau$,
  cf.~\cite[\S2.1, Proposition~2]{Fulton-toric}, so that $A_\sigma$ is
  a localisation of~$A_\tau$. Since localisation is exact both
  vertical maps in the following square diagram are quasi-isomorphisms:
  \begin{diagram}
    A^\sigma \tensor_{A^\tau} C^\tau & \rTo & C^\sigma \\
    \dTo && \dTo \\
    A^\sigma \tensor_{A^\tau} D^\tau & \rTo & D^\sigma
  \end{diagram}
  This proves that the upper horizontal map is a quasi-isomorphism if
  and only if the lower horizontal map is a quasi-isomorphism.
\end{proof}

\begin{lemma}
  \label{lem:homotopy_sheaf_2_of_3}
%  \label{lem:kernel_homotopy_sheaves}
  Suppose we have a short exact sequence
  \[0 \rTo B \rTo C \rTo D \rTo 0\]
  of objects in~$\pre(\fan)$. Then if two of the three presheaves $B$,
  $C$ and~$D$ are homotopy sheaves, so is the third.
\end{lemma}

\begin{proof}
  Let $\sigma \subseteq \tau$ be an inclusion of cones in~$\fan$.
  Consider the following commutative ladder diagram:
  \begin{diagram}
    0 & \rTo & A^\sigma \tensor_{A^\tau} B^\tau & \rTo[l>=3em] &
    A^\sigma \tensor_{A^\tau} C^\tau & \rTo[l>=3em] & A^\sigma
    \tensor_{A^\tau} D^\tau & \rTo & 0\\
    && \dTo && \dTo && \dTo \\
    0 & \rTo & B^\sigma & \rTo & C^\sigma & \rTo & D^\sigma &
    \rTo & 0
  \end{diagram}
  The bottom row is exact by hypothesis. Since $A^\sigma$ is a
  localisation of $A^\tau$ the top row is exact as well. Moreover, by
  hypothesis two of the vertical maps are quasi-isomorphisms. The five
  lemma, applied to the associated infinite ladder diagram of homology
  modules, guarantees that the third vertical map is a
  quasi-isomorphism as well.
\end{proof}

Since a retract of a quasi-isomorphism is a quasi-isomorphism, we also
have:

\begin{lemma}
  \label{lem:hom_sheaf_inv_retracts}
  Suppose that $C$ is a retract, in the category~$\pre(\fan)$, of the
  homotopy sheaf~$D$. Then $C$~is a homotopy sheaf. \qed
\end{lemma}

% \begin{proof}
%   Let $\sigma \rTo \tau$ be an inclusion of cones in~$\fan$. Since
%   $D$~is a homotopy sheaf, the map $A^\sigma \tensor_{A^\tau} D^\tau
%   \rTo D^\sigma$ is a quasi-isomorphism. Hence its retract $A^\sigma
%   \tensor_{A^\tau} C^\tau \rTo C^\sigma$ is a quasi-isomorphism as
%   well, proving that $C$~is a homotopy sheaf.
% \end{proof}

\begin{proposition}
  Let $\rho$ be a $1$-cone in~$\fan$.
  \begin{enumerate}
  \item[{\rm (1)}] The restriction functor $\varepsilon \colon \pre(\fan) \rTo
    \pre(\fan/\rho)$, defined in \S\ref{par:restriction}, is a left
    \textsc{Quillen} functor with respect to the $c$-structure
    (\ref{sec:c-structure}).
  \item[{\rm (2)}] The functor $\varepsilon$ preserves strict sheaves and $f$-cofibrant
    (\ref{sec:f-structure}) homotopy sheaves.
  \end{enumerate}
\end{proposition}

\begin{proof}
  Part~(1) is true since the right adjoint~$\zeta$ of~$\varepsilon$
  clearly preserves fibrations and acyclic fibrations in the
  $c$-structure.

  For (2) suppose that $C \in \pre(\fan)$ is a strict sheaf. An
  inclusion of cones $\bar \sigma \subseteq \bar \tau$ in $\fan/\rho$
  corresponds to an inclusion of cones $\sigma \subseteq \tau$ in
  $\st(\rho)$. The commutative diagram
  \begin{diagram}
    A^\sigma & \rTo & A^{\bar \sigma} \\
    \uTo && \uTo \\
    A^\tau & \rTo & A^{\bar \tau}
  \end{diagram}
  then induces the top horizontal isomorphism in the following diagram:
  \begin{equation}
    \label{eq:diagram_sheaf_cond}
      \begin{diagram}
        A^{\bar\sigma} \tensor_{A^{\bar\tau}}
        \varepsilon(C)^{\bar\tau} &=& A^{\bar\sigma}
        \tensor_{A^{\bar\tau}} A^{\bar\tau} \tensor_{A^\tau} C^\tau &
        \rTo[l>=4em]^\iso & A^{\bar\sigma} \tensor_{A^\sigma}
        A^{\sigma} \tensor_{A^\tau}
        C^\tau \\
        \dTo &&&& \dTo>\iso \\
        \varepsilon(C)^{\bar\sigma} && \rTo_= && A^{\bar\sigma}
        \tensor_{A^\sigma} C^\sigma
      \end{diagram}
  \end{equation}
  The right vertical map is an isomorphism as $C$ is a strict sheaf.
  Hence the left vertical map is an isomorphism as well, which proves
  that $\varepsilon(C)$ is a strict sheaf as claimed.

  Now suppose that $C$ is an $f$-cofibrant homotopy sheaf. We want to
  prove that $\varepsilon (C)$ is an $f$-cofibrant homotopy sheaf as
  well. Fix $\sigma \in st(\rho)$. Since $C$ is $f$-cofibrant we know
  that $C^\sigma$ is cofibrant in the category of $A^\sigma$-module
  chain complexes. Hence $\varepsilon (C)^{\bar \sigma} = A^{\bar
    \sigma} \tensor_{A^\sigma} C^\sigma$ is cofibrant in the category
  of $A^{\bar \sigma}$-module chain complexes. As this is true for all
  $\sigma \in \st(\rho)$ we know that $\varepsilon(C)$ is
  $f$-cofibrant. We are left to check that for all $\sigma \subseteq
  \tau$ in $\st(\rho)$ the left vertical map in the
  diagram~(\ref{eq:diagram_sheaf_cond}) is a weak equivalence. By
  hypothesis, the map $A^\sigma \tensor_{A^\tau} C^\tau \rTo C^\sigma$
  is a weak equivalence of cofibrant objects. Hence the right vertical
  map of diagram~(\ref{eq:diagram_sheaf_cond}), obtained by base
  change, is a weak equivalence as well, proving the assertion.
\end{proof}

\subsection{Colocal objects and colocal equivalences}

\begin{notation}
  For $\vec k \in \bZ^r$ and $\ell\in \bZ$ we let $\mathcal{O}(\vec
  k)[\ell]$, cf.~\S\ref{def:twisting_sheaf}, denote the sheaf
  $\mathcal{O}(\vec k)$ considered as a chain complex concentrated in
  chain degree~$\ell$. We denote by $\hco(\vec k)$ the
  $c$-cofibrant replacement $\hco(\vec k) \rFib^\sim
  \mathcal{O}(\vec k)$ with source consisting of bounded chain complexes of
  finitely generated free modules;
  more specifically, we use a mapping cylinder
  factorisation construction dual to the canonical path space
  factorisation discussed earlier. Note that
  $\hco(\vec k)[\ell] \rTo \mathcal{O}(\vec k)[\ell]$ then is a
  $c$-cofibrant replacement as well with source a strict sheaf in the
  sense of Definition~\ref{def:sheaves}.

  \medbreak

  For a given chain complex~$M$ of $A$-bimodules, we define the
  presheaf
  \[M \tensor \mathcal{O}(\vec k)[\ell] \colon \quad \sigma
  \mapsto M \tensor_A \mathcal{O}(\vec k)[\ell]^\sigma\ ,\]
  and similarly for $\hco(\vec k)[\ell]$. The resulting presheaves are
  in fact strict sheaves as is easily checked by inspection.
\end{notation}

\begin{definition}
  \label{def:colocal_equivalence_one}
  A map $f \colon C \rTo D$ in $\pre(\fan)$ is called an {\it
    $\hco(\vec k)[\ell]$-colocal equivalence\/},
  cf.~\cite[Definition~3.1.8~(1)]{Hirsch-Loc}, if the induced map
  \[\hom_{\pre(\fan)}(NA[\Delta^\bullet] \tensor \hco(\vec k)[\ell]
  ,\, C) \rTo \hom_{\pre(\fan)}(NA[\Delta^\bullet] \tensor \hco(\vec
  k)[\ell] ,\, D)\]
  is a weak homotopy equivalence of simplicial sets.  Here
  $NA[\Delta^\bullet]$ is the cosimplicial $A$-bimodule chain complex
  $n \mapsto NA[\Delta^n]$ with $N$ the reduced chain complex functor.
\end{definition}

\begin{proposition}
  \label{prop:holim_detects_colocal}
  Fix $\ell \in \bZ$ and $\vec k \in \bZ^r$. A map $f \colon C \rTo D$
  of objects in $\pre(\fan)$ is an $\hco(\vec k)[\ell]$-colocal
  equivalence if and only if the corresponding map of $A$-module chain complexes
  \[\holim\, C(-\vec k) \rTo \holim\, D(-\vec k)\]
  induces isomorphisms on homology in degrees $\geq \ell$.
\end{proposition}

\begin{proof}
  Let $C \rTo^\sim PC$ denote the canonical $f$-fibrant replacement
  for~$C$, cf.~\ref{canonical-f-replacement}, and recall that $\holim\,
  C = \lim PC$. Similarly, we have a weak equivalence $D \rTo^\sim
  PD$. The map $f$ induces a corresponding map $\tilde f \colon PC
  \rTo PD$.  Consider the huge diagram of Fig.~\ref{fig:diag1}.
  \begin{figure}[ht]
    \small
    \begin{diagram}[labelstyle=\scriptstyle]
      \hom_{\pre(\fan)} \big(NA [\Delta^\bullet] \tensor \hco(\vec k)[\ell], \, C\big) & \rTo^{f_*}
      & \hom_{\pre(\fan)} \big(NA [\Delta^\bullet] \tensor \hco(\vec k)[\ell], \, D\big) \\
      \dTo<\sim & 1 & \dTo>\sim \\
      \hom_{\pre(\fan)} \big(NA [\Delta^\bullet] \tensor \hco(\vec k)[\ell], \, PC\big) & \rTo
      ^{\tilde f_*}
      & \hom_{\pre(\fan)} \big(NA [\Delta^\bullet] \tensor \hco(\vec k)[\ell], \, PD\big) \\
      \uTo<\sim & 2 & \uTo>\sim \\
      \hom_{\pre(\fan)} \big(NA [\Delta^\bullet] \tensor \co(\vec k)[\ell], \, PC\big) & \rTo
      ^{\tilde f_*}
      & \hom_{\pre(\fan)} \big(NA [\Delta^\bullet] \tensor \co(\vec k)[\ell], \, PD\big) \\
      \dTo<\iso & 3 & \dTo>\iso \\
      {\begin{diagram}[small] \hom_{Ch_A} \Big(NA [\Delta^\bullet], {\hbox to 4em{\hfil}} \\
              {\hbox to 6em{\hfil}} \underline{\hom} (\co(\vec k)[\ell], \, PC)
              \Big) \end{diagram}} &
      \rTo ^{\tilde f_*}
      & {\begin{diagram}[small] \hom_{Ch_A} \Big(NA [\Delta^\bullet], {\hbox to 4em{\hfil}} \\
                {\hbox to 6em{\hfil}} \underline{\hom} (\co(\vec
                k)[\ell], \, PD)\Big)\end{diagram}} \\
      \dTo<\iso & 4 & \dTo>\iso \\
      {\begin{diagram}[small] \hom_{Ch_A} \Big(NA [\Delta^\bullet], {\hbox to
          4em{\hfil}} \\
          {\hbox to 6em{\hfil}} \lim (PC)(-\vec
          k)[-\ell]\Big)\end{diagram}} &
            \rTo^{\tilde f_*}
      & {\begin{diagram}[small] \hom_{Ch_A} \Big(NA [\Delta^\bullet], {\hbox to
            4em{\hfil}} \\
            {\hbox to 6em{\hfil}} \lim (PD)(-\vec k)[-\ell]\Big)\end{diagram}} \\
      \dTo<\sim & 5 & \dTo>\sim \\
      {\begin{diagram}[small] \hom_{Ch_A} \Big(NA [\Delta^\bullet], {\hbox to
          4em{\hfil}} \\
          {\hbox to 6em{\hfil}} \holim\, C(-\vec k)[-\ell]\Big)\end{diagram}} &
            \rTo^{\tilde f_*}
      & {\begin{diagram}[small] \hom_{Ch_A} \Big(NA [\Delta^\bullet], {\hbox to
            4em{\hfil}} \\
            {\hbox to 6em{\hfil}} \holim\, D(-\vec k)[-\ell]\Big)\end{diagram}} \\
      \dTo<= & 6 & \dTo>= \\
      W (\holim\, C(-\vec k)[-\ell]) & \rTo^{W(\holim f(-\vec k))} & W (\holim\, D(-\vec k)[-\ell])
    \end{diagram}
    \caption{Diagram}
    \label{fig:diag1}
  \end{figure}
  We claim that the vertical maps are weak equivalences or
  isomorphisms of simplicial sets as marked. We list the reasons for
  each of the squares:

  \smallskip \noindent {\it Square~1:\/} We know that $NA
  [\Delta^\bullet] \tensor \hco(\vec k)$ is a cosimplicial resolution
  of $\hco(k)$ with respect to the $c$-structure of~$\pre(\fan)$, and
  that $C$, $PC$, $D$ and~$PD$ are $c$-fibrant. It follows from
  \cite[Corollary~16.5.5~(2)]{Hirsch-Loc} that the vertical maps are
  weak equivalences.
  
  \smallskip \noindent {\it Square~2:\/} This follows immediately from
  \cite[Corollary~16.5.5~(1)]{Hirsch-Loc} since $PC$ and $PD$ are
  $f$-fibrant by construction, and since the map
  \[NA[\Delta^\bullet] \tensor \mathcal{O}(\vec k)[\ell] \rTo
  NA[\Delta^\bullet] \tensor \hco(\vec k)[\ell]\]
  is a \textsc{Reedy\/} weak equivalence of cosimplicial resolutions
  for the $f$-structure of~$\pre(\fan)$.

  \smallskip \noindent {\it Square~3:\/} Use adjointness of tensor
  product and hom complex for each entry of the diagrams involved.
  Note that $\co (\vec k)[\ell]$ is a chain-complex with non-trivial
  entries in degree~$\ell$ only.

  \smallskip \noindent {\it Square~4:\/} This uses the isomorphism of
  functors from Lemma~\ref{lem:hom_lim_iso}.

  \smallskip \noindent {\it Square~5:\/} Recall that $C \rTo PC$ is an
  $f$-fibrant replacement, hence so is its $(-\vec k)$th twist
  $C(-\vec k) \rTo (PC) (-\vec k)$ by
  Corollary~\ref{cor:twist_is_Quillen}.  But
  \[C(-\vec k) \rTo P(C(-\vec k))\]
  is another $f$-fibrant replacement, so we know that
  $(PC) (-\vec k)$ and $P(C(-\vec k))$ are weakly equivalent. Since
  both objects are $f$-fibrant they are fibrant as diagrams of
  $A$-module chain complexes. In particular, application of the
  inverse limit functor yields weakly equivalent chain complexes. The
  left vertical map then is known to be a weak equivalence by
  \cite[Corollary~16.5.5~(1)]{Hirsch-Loc}, applied to the category
  $\mathrm{Ch}_A$ with the projective model structure; for the target, note
  that $\lim P(C(-\vec k)) = \holim\, C(-\vec k)$ by definition of
  homotopy limits.---A similar argument applies to the right vertical
  map.

  \smallskip \noindent {\it Square~6:\/} This is just the definition
  of the \textsc{Dold}-\textsc{Kan} functor~$W$.

  \smallskip In particular, $f$ is an $\hco(\vec k)[\ell]$-colocal
  equivalence if and only if the top horizontal map $f_*$ is a weak
  equivalence if and only if $W(\holim f(-\vec k))$ is a weak
  equivalence if and only if $\,\holim f(-\vec k)[-\ell]$ is a
  quasi-isomorphism in non-negative degrees.
\end{proof}

\begin{definition}
  \label{def:colocal}
  Let $R \subseteq \bZ^{\fan(1)}$ be a non-empty subset.
  \begin{enumerate}
  \item[{\rm (1)}] A map $f \in \pre(\fan)$ is called an {\it $R$-colocal
      equivalence\/} if it is an $\hco (\vec k)[\ell]$-colocal
    equivalence in the sense of
    Definition~\ref{def:colocal_equivalence_one} for all $\vec k \in
    R$ and $\ell \in \bZ$. In other words, $f$ is an $R$-colocal
    equivalence if and only if it is a colocal equivalence in the
    sense of \cite[Definition~3.1.8~(1)]{Hirsch-Loc} with respect to
    the set $\hco (R) := \{ \hco (\vec k)[\ell] \,|\, \vec k \in R,\,
    \ell \in \bZ \}$.
  \item[{\rm (2)}] An object $B \in \pre(\fan)$ is called {\it $R$-colocal\/} if
    it is $\hco(R)$-colocal in the sense of
    \cite[Definition~3.1.8~(2)]{Hirsch-Loc} with respect to the
    $c$-structure of~$\pre(\fan)$; equivalently, if $B$ is
    $c$-cofibrant and $\hco(R)$-cellular
    \cite[Theorem~5.1.5]{Hirsch-Loc}.
  \end{enumerate}
  If the set~$R$ is understood we will drop it from the notation and
  simply speak of colocal equivalences and colocal objects.
\end{definition}

More explicitly, a map $f \colon C \rTo D$ in $\pre(\fan)$ is an
$R$-colocal equivalence if for all $\vec k \in R$ and all $\ell \in
\bZ$ the map
\[\hom_{\pre(\fan)} \big( NA [\Delta^\bullet] \tensor \hco (\vec
k)[\ell],\, C\big) \rTo^{f_*} \hom_{\pre(\fan)} \big( NA
[\Delta^\bullet] \tensor \hco(\vec k)[\ell],\, D\big)\]
is a weak equivalence of simplicial sets. The object $B \in
\pre(\fan)$ is $R$-colocal if it is $c$-cofibrant, and if for all
$R$-colocal maps $f \colon C \rTo D$ in $\pre(\fan)$ the map
\[\hom_{\pre(\fan)} \big( {\bf B}, C \big) \rTo \hom_{\pre(\fan)}
\big( {\bf B}, D \big)\]
is a weak equivalence of simplicial sets, where ${\bf B}$ denotes a
cosimplicial resolution \cite[Definition~16.1.20~(1)]{Hirsch-Loc}
of~$B$ with respect to the $c$-structure of~$\pre(\fan)$.

\begin{corollary}
  \label{cor:holim_detects_colocal}
  A map $f$ in $\pre(\fan)$ is an $\hco(\{\vec k\})$-colocal
  equivalence if and only if $\holim(f(-\vec k))$ is a quasi-isomorphism.
\end{corollary}

\begin{proof}
  This follows from Proposition~\ref{prop:holim_detects_colocal},
  together with the fact that a map $g$ of chain complexes is a weak
  equivalence if and only if $W(g[\ell])$ is a weak equivalence of
  simplicial sets for all $\ell \in \bZ$.
\end{proof}

\subsection{Colocally acyclic objects}

\begin{definition}
  \label{def:colocally_acyclic}
  Let $R \subseteq \bZ^{\fan(1)}$ be a non-empty subset.  An object $B
  \in \pre(\fan)$ is called {\it $R$-colocally acyclic\/} if the
  unique map $B \rTo 0$ is an $R$-colocal equivalence.  If the set~$R$
  is understood we will drop it from the notation and simply speak of
  colocally acyclic objects.
\end{definition}

\begin{notation}
  \begin{enumerate}
  \item[{\rm (1)}] For a cone $\sigma \in \fan$ let $\vec \sigma \in
    \bZ^{\fan(1)}$ denote the vector whose $\rho$-component is~$1$ if
    $\rho \subseteq \sigma$, and is~$0$ otherwise. Note that the
    zero-cone corresponds to the zero-vector.
  \item[{\rm (2)}] Similarly, we write $-\vec\sigma$ for the vector whose
    $\rho$-component is~$-1$ if $\rho \subseteq \sigma$, and is~$0$
    otherwise.
  \end{enumerate}
\end{notation}

\begin{construction}
  \label{constr:R_fan}
  To the regular fan~$\fan$ we associate a finite set $R_\fan \subset
  \bZ^{\fan(1)}$ as follows:
  \begin{enumerate}
  \item[{\rm (1)}] If $\fan$ has a unique inclusion-maximal cone (so $X_\fan$ is
    affine), we set $R_\fan := \{\vec 0\}$. This covers the unique fan
    in~$\bR^0$ as a special case.
  \item[{\rm (2)}] Suppose that $\fan$ does not have a unique inclusion-maximal
    cone. Let $\rho \in \fan(1)$ be a $1$-cone. We consider
    $\bZ^{(\fan/\rho) (1)}$ as a subset of $\bZ^{\fan(1)}$ in the
    following way: A $1$-cone $\bar\sigma \in \fan/\rho$ corresponds
    to a $2$-cone $\sigma \in \fan$ which contains exactly two
    $1$-cones: The cone $\rho$ and a cone $\tau \not=\rho$. We
    identify the $\bar\sigma$-component of $\bZ^{(\fan/\rho) (1)}$
    with the $\tau$-component of~$\bZ^{\fan(1)}$. All other components
    will be set to~$0$.---Using this identification, we set
    \[R_\fan := \bigcup_{\rho \in \fan(1)} R_{\fan/\rho} \ \cup \ 
    \bigcup_{\rho \in \fan(1)} \big(\vec\rho + R_{\fan/\rho} \big) \]
    where $\vec\rho + R_{\fan/\rho} = \{ \vec\rho + \vec k \,|\, \vec
    k \in R_{\fan/\rho}\}$. Note that $R_{\fan/\rho}$ is defined by
    induction on the dimension of~$\fan$.
  \end{enumerate}
\end{construction}

\begin{example}
  \label{example:complete_fan}
  If $\fan$ is complete then $R_\fan = \{\vec \sigma \,|\, \sigma \in
  \fan \}$.
\end{example}

\begin{proposition}
  \label{prop:null_implies_acyclic}
  If $C \in \pre(\fan)$ is an $R_\fan$-colocally acyclic $c$-cofibrant
  homotopy sheaf on~$X_\fan$, then $C \simeq 0$ in the $c$-structure
  (\ie, all complexes $C^\sigma$ are
  acyclic). See~\ref{constr:R_fan} for a definition
  of~$R_\fan$.
\end{proposition}

\begin{proof}
  The statement is true if the fan $\fan$ contains a unique
  inclusion-maximal cone~$\mu$ (so $X_\fan = U_\mu$ is affine).
  Indeed, by Remark~\ref{remark:holim} we have a quasi-isomorphism
  $C^\mu \rTo \mathrm{holim}(C)$. If $C$ is $R_\fan$-colocally
  acyclic, then $\mathrm{holim}(C) \simeq 0$ (since $\vec 0 \in
  R_\fan$), hence $C^\mu \simeq 0$. Since $C$~is a homotopy sheaf,
  this implies that all its components $C^\tau \simeq A^\tau
  \tensor_{A^\mu} C^\mu$ are acyclic as well.---In particular, the
  Proposition is true for the unique fan in~$\bR^0$.

  \smallbreak  

  If $\fan$ does not contain a unique inclusion-maximal cone, we
  proceed by induction on the dimension.

  {\bf Induction hypothesis:\/} The theorem holds for objects of
  $\pre(\Delta)$ for all regular fans~$\Delta$ with $\dim \Delta <
  \dim \fan = n$.

  \smallbreak
  
  {\bf Step~1: The map $C (-\vec\rho) \rTo C(\vec 0) \iso C$ is a weak
    equivalence for each $\rho \in \fan(1)$.} Fix a $1$-cone $\rho \in
  \fan$, and fix $\vec k \in R_{\fan/\rho} \subset R_\fan$, the
  inclusion of sets as explained in
  Construction~\ref{constr:R_fan}~(2).  Then $\vec\rho + \vec k \in
  R_\fan$ by construction.

  The inclusion ${\cal O}\big(-(\vec\rho +
  \vec k)\big) \rTo\relax {\cal O} (-\vec k)$ induces a short exact sequence
  of objects in~$\pre(\fan)$
  \begin{equation}
    \label{eq:ses_ind}
    0 \rTo C \big(-(\vec\rho + \vec k)\big) \rTo^i C(-\vec k) \rTo Q
    (-\vec k) \rTo 0 \ .
  \end{equation}
  If $\sigma$~is a cone not containing~$\rho$ then all components of
  the vectors $\vec k$ and~$\vec\rho + \vec k$ corresponding to
  $1$-cones in~$\sigma(1)$ vanish. Hence the $\sigma$-component of the
  inclusion~$i$ is the identity, so that $Q(-\vec k)^\sigma = 0$ in
  this case.

  From the above sequence we obtain a short exact sequence of
  $A$-module chain complexes
  \[0 \rTo \holim\, C \big(-(\vec\rho + \vec k)\big) \rTo \holim\,
  C(-\vec k) \rTo \holim\, Q(-\vec k) \rTo 0 \ .\]
  Now since $C \rTo 0$ is an $R_\fan$-colocal equivalence by
  hypothesis, Corollary~\ref{cor:holim_detects_colocal} (applied to
  the vectors $\vec k$ and $\vec\rho +\vec k$ in~$R_\fan$) yields that
  \[\holim\, C\big(-(\vec\rho + \vec k)\big) \simeq 0 \simeq \holim\,
  C(-\vec k) \ .\]
  We conclude that $\holim\, Q(-\vec k) \simeq 0$ as well.

  From Proposition~\ref{cofibre-of-inclusion} it is easy to conclude
  that $Q(-\vec k) = \zeta \big(\varepsilon (C(-\vec k))\big)$ is
  nothing but the extension by zero of the restriction~$C(-\vec
  k)|_{V_\rho}$ of~$C (-\vec k)$ to $V_\rho = X_{\fan/\rho}$. Since
  twisting commutes with restriction, $Q(-\vec k)$ could equally be
  described as the extension by zero of the $(-\vec k)$th twist of the
  restriction~$C|_{V_\rho}$.

  \smallskip

  In other words, we have shown that for all $\vec k \in
  R_{\fan/\rho}$ the chain complex
  $\holim_{(\fan/\rho)^\op} C|_{V_\rho} (-\vec k)$ is acyclic
  where we have used Lemma~\ref{lem:holim_of_extension} to restrict to
  the smaller indexing category $\st(\rho)^\op$ in the homotopy limit.
  From Corollary~\ref{cor:holim_detects_colocal} we infer that the
  map $C|_{V_\rho} \rTo *$ in $\pre(\fan/\rho)$ is an
  $R_{\fan/\rho}$-colocal equivalence. But by the induction hypothesis
  we then know that $C|_{V_\rho} \simeq 0$. Since $Q(\vec 0) = \zeta
  (C|{V_\rho})$ this implies that $Q (\vec 0) \simeq 0$. From the
  short exact sequence~(\ref{eq:ses_ind}), applied to $\vec k = \vec 0
  \in R_{\fan/\rho}$ we then see that the map $C(-\vec\rho) \rTo C(\vec
  0) \iso C $ is a
  weak equivalence as claimed.

  \smallbreak

  {\bf Step~2: All the structure maps $C^\sigma \rTo C^\tau$ of~$C$ are
    quasi-iso\-mor\-phisms.} Let $\tau \subset \sigma$ be a codimension-$1$
  inclusion of cones in~$\fan$. Let $\rho$ denote the unique $1$-cone
  contained in~$\sigma \setminus \tau$. We want to identify the
  $\sigma$-component of the first map in the
  sequence~(\ref{eq:ses_ind}) for $\vec k = \vec 0$: By definition, it
  is the natural inclusion map
  \begin{equation}
    \label{eq:incl_map1}
    {\cal O}(-\vec\rho)^\sigma \tensor_{A^\sigma} C^\sigma \rTo\relax
    {\cal O}(\vec 0)^\sigma \tensor_{A^\sigma} C^\sigma \iso C^\sigma \ .
  \end{equation}
  Since $\fan$ is regular we can choose $f \in M$ such that $f$
  vanishes on the primitive generators of~$\tau$, and such that
  $f$~takes the value~$1$ on the primitive generator of~$\rho$. Then
  $f \in S_\sigma$, and there is an isomorphism of $A^\sigma$-modules
  ${\cal O} (\vec 0)^\sigma \rTo\relax {\cal O}(-\vec\rho)^\sigma$
  described by $b \mapsto b + f$ on elements of the canonical
  $A$-basis. We can thus rewrite the map~(\ref{eq:incl_map1}) up to
  isomorphism as
  $C^\sigma \rTo^{f} C^\sigma$.

  The module chain complex $A^\tau \tensor_{A^\sigma} C^\sigma$ is obtained
  from~$C^\sigma$ by inverting the action of the element~$f$
  (Lemma~\ref{lem:monoid_generators}), \ie, by forming the colimit of the
  sequence
  \[C^\sigma \rTo^f C^\sigma \rTo^f C^\sigma \rTo^f \ldots \ .\] Now
  $f$ acts by quasi-isomorphism on~$C^\sigma$ by the results of
  Step~1; indeed, as just seen above $f$ is the $\sigma$-component of
  the weak equivalence $C(-\vec\rho) \rTo C(\vec 0) \iso C$. Hence the
  canonical map $C^\sigma \rTo A^\tau \tensor_{A^\sigma} C^\sigma$ is
  a quasi-isomorphism. Since $C$~is a homotopy sheaf, the map $A^\tau
  \tensor_{A^\sigma} C^\sigma \rTo C^\tau$ is a quasi-isomorphism. The
  combination of these two statements shows that the structure map
  $C^\sigma \rTo C^\tau$ is a quasi-isomorphism.

  As any inclusion of cones in~$\fan$ can be written as a sequence of
  codimension\nobreakdash-$1$ inclusions, it follows that all
  structure maps of~$C$ are quasi-isomorphisms as claimed.

  \smallbreak

  {\bf Step~3: All entries of the diagram~$C$ are acyclic.}  Write
  $\mathrm{con}(B)$ for the constant $\fan^\mathrm{op}$-diagram with
  value~$B$. Fix a cone $\sigma \in \fan$. The structure maps of~$C$
  assemble to maps of diagram
  \[C \rTo \mathrm{con} (C^{\{0\}}) \lTo \mathrm{con}(C^\sigma) \ ;\]
  both these maps are weak equivalences of diagrams of
  $A$-module chain complexes by Step~2.  Application of the homotopy
  limit functor gives a chain of quasi-isomorphisms (we use
  Lemma~\ref{lem:holim_constant} in the last step)
  \[\holim\, C \rTo^\sim \holim\, \mathrm{con}
  (C^{\{0\}}) \lTo^\sim \holim\, \mathrm{con}(C^\sigma) \simeq
  C^\sigma \ .\]
  But since $\vec 0 \in R_\fan$ we know by
  Corollary~\ref{cor:holim_detects_colocal} that $\holim\, C \simeq
  0$, so $C^\sigma \simeq 0$ as required.
\end{proof}

\subsection{Homotopy sheaves as cofibrant objects}

\begin{proposition}[Colocal model structure of $\pre(\fan)$]
  \label{prop:colocal_model_structure}
  Let $R \subseteq \bZ^{\fan(1)}$. The category $\pre(\fan)$ has a
  model structure, called the {\em $R$-colocal model structure}, where
  a map~$f$ is a weak equivalence if and only if it is an $R$-colocal
  equivalence (Definition~\ref{def:colocal}), and a fibration if and
  only if it is a fibration in the $c$-structure
  of\/~$\pre(\fan)$. The model structure is right proper, and every
  object is fibrant.
\end{proposition}

\begin{proof}
  This is \cite[Theorem~5.1.1]{Hirsch-Loc}, applied to the
  $c$-structure of~$\pre(\fan)$.
\end{proof}

\begin{theorem}
  \label{thm:homsheaf_colocal}
  Let $R_\fan \subset \bZ^{\fan(1)}$ denote the finite set specified in
  Construction~\ref{constr:R_fan}.
  \begin{enumerate}
    \item[{\rm (1)}] If $C$ is an $R_\fan$-colocal object of
        $\pre(\fan)$, then $C$~is a $c$-cofibrant homotopy sheaf.
    \item[{\rm (2)}] If $C$ is a $c$-cofibrant homotopy sheaf
    on~$X_\fan$, then $C$ is $R_\fan$-colocal.
  \end{enumerate}
\end{theorem}

\begin{proof} We consider the category $\pre(\fan)$ equipped with the
  $R_\fan$-colocal model structure of
  Proposition~\ref{prop:colocal_model_structure}.

  \smallbreak

  Part~(1) follows from the description of colocal objects in
  the general theory of right \textsc{Bousfield} localisation. We have
  to introduce some auxiliary notation and results first.

  Recall that the $c$-structure of~$\pre(\fan)$ has a set
  \[J_c = \{F_\tau (0 \rTo D_n (A)) \,|\, n \in \bZ,\, \tau \in \fan\}\]
  of generating cofibrations as specified in
  Lemma~\ref{lem:c_structure_is_cellular}. Since the chain complexes
  $D_n (A)$ are acyclic so are all the entries in the
  diagrams~$F_\tau(D_n(A))$. Consequently, all maps in~$J_c$ are
  injective maps of homotopy sheaves, and their cofibres are homotopy
  sheaves.

  The set
  \[\Lambda (R_\fan) := \{ L_n (NA[\Delta^\bullet]) \tensor
  \hco(\vec k) \rTo\relax NA[\Delta^n] \tensor \hco (\vec k)
  \ |\ \vec k \in R_\fan\}\]
  is a full set of horns on~$\hco (R_\fan)$ in the sense of
  \cite[Definition~5.2.1]{Hirsch-Loc}; here $N$~denotes the reduced
  chain complex functor as usual. This follows from the fact that
  $NA[\Delta^\bullet] \tensor \hco (\vec k)$ is a cosimplicial
  resolution of~$\hco(\vec k)$ by
  Lemma~\ref{lem:cosimplicial_resolution_pre_fan}. Note that
  $\Lambda(R_\fan)$ is a set of injective maps of homotopy sheaves;
  the cofibres are the objects
  \[N A[\Delta^n/\partial \Delta^n] \tensor \hco(\vec k)\ \quad n
  \geq 0,\ \vec k \in R_\fan\]
  which are homotopy sheaves as well.
  
  Now suppose that $C$~an $R_\fan$-colocal object
  of~$\pre(\fan)$. From \cite[Corollary~5.3.7]{Hirsch-Loc} we know
  that $C$ is a retract of a $c$-cofibrant object $X \in \pre(\fan)$
  which admits a weak equivalence $X \rTo^\sim Y$ to an object $Y \in
  \pre(\fan)$ which is a cell complex with respect to the maps in $J_c
  \cup \Lambda (R_\fan)$.  Since the cofibres of all the maps in this
  set are homotopy sheaves as observed above, it follows from
  (transfinite) induction on the number of cells in~$Y$ that the
  presheaf~$Y$ is a homotopy sheaf. The induction step works as
  follows: Suppose that $f \colon A \rTo B$ is an injective map of
  presheaves such that its cofibre~$B/A$ is a homotopy sheaf, and
  suppose that $Z$~is a homotopy sheaf. Then there is a short exact
  sequence in~$\pre(\fan)$
  \[0 \rTo Z \rTo Z \mathop{\cup}_A B \rTo B / A \rTo 0\]
  where $Z$ and~$B/A$ are homotopy sheaves. It follows from
  Lemma~\ref{lem:homotopy_sheaf_2_of_3} that $Z \cup_A B$ is a
  homotopy sheaf as well.

  Since $Y$ is a homotopy sheaf so is the presheaf~$X$
  by Lemma~\ref{lem:homotopy_invariant}; consequently, its retract~$C$
  is a homotopy sheaf as well
  (Lemma~\ref{lem:hom_sheaf_inv_retracts}).

  \smallbreak

  Part~(2): Let $\tilde Y \rFib^{\coweq} C$ be a cofibrant replacement
  with respect to the colocal model structure, constructed by
  factorising the map $0 \rTo C$ as a colocally acyclic cofibration
  followed by a $c$-fibration. Then $Y$ is $R$-colocal. We will show
  that the map $Y \rFib^\coweq C$ is a weak equivalence (in the
  $c$-structure); then $C$ is colocal as well by
  \cite[Proposition~3.2.2~(2)]{Hirsch-Loc}.
  
  The map $Y \rFib^\coweq C$ is a $c$-fibration, hence surjective. We
  thus have a short exact sequence of objects in $\pre(\fan)$
  \begin{equation}
    \label{eq:colocal_seq}
    0 \rTo \tilde K \rTo Y \rFib^\coweq C \rTo 0 \ .
  \end{equation}
  The map $\tilde K \rTo 0$ is the pullback of $Y \rFib^\coweq C$, so
  $\tilde K \rTo 0$ is a colocally acyclic fibration, hence $\tilde K$
  is colocally acyclic. By considering the long exact homology
  sequence associated to~(\ref{eq:colocal_seq}) we are reduced to
  showing $\tilde K \simeq 0$. Let $K \rFib^\sim \tilde K$ denote a
  $c$-cofibrant replacement. It is enough to prove that $K \simeq 0$.
  Note that $K$ is $R_\fan$-colocally acyclic since $\tilde K$ is so,
  and since every weak equivalence is a colocal equivalence
  \cite[Proposition~3.1.5]{Hirsch-Loc}.
  
  By hypothesis and part~(1), both $Y$ and~$C$ are homotopy sheaves.
  Hence $\tilde K$, being the kernel of a surjection $Y \rTo C$, is a
  homotopy sheaf as well by Lemma~\ref{lem:homotopy_sheaf_2_of_3}.
  Consequently, $K$ is a $c$-cofibrant homotopy sheaf which satisfies
  the hypotheses of Proposition~\ref{prop:null_implies_acyclic} which
  proves $K \simeq 0$ as required.
\end{proof}

\begin{corollary}
  \label{cor:weqs_of_homotopy_sheaves}
  Let $R_\fan \subset \bZ^{\fan(1)}$ denote the finite set specified
  in Construction~\ref{constr:R_fan}. Let $f \colon X \rTo Y$ be a map
  of homotopy sheaves. Then $f$~is a weak equivalence if and only if
  the induced map of chain complexes
  \[\holim f(-\vec k) \colon \holim X(-\vec k) \rTo \holim Y(-\vec
  k)\]
  is a quasi-isomorphism for all $\vec k \in R_\fan$.
\end{corollary}

\begin{proof}
  By $c$-cofibrant approximation and lifting, we can construct a
  commutative square
  \begin{equation}
    \label{eq:square_colocal}
    \begin{diagram}
      X^c & \rTo^{f^c} & Y^c \\
      \dTo<\sim && \dTo>\sim \\
      X & \rTo_f & Y
    \end{diagram}
  \end{equation}
  where both vertical maps are weak equivalences, and with
  $c$-cofibrant presheaves $X^c$ and~$Y^c$. Then $f$~is a weak
  equivalence if and only if $f^c$~is. Now $f^c$ is a map of
  $R_\fan$-colocal objects by Lemma~\ref{lem:homotopy_invariant} and
  Theorem~\ref{thm:homsheaf_colocal}. Hence $f^c$~is a weak
  equivalence if and only if $f^c$~is an $R_\fan$-colocal map
  \cite[Theorem~3.2.13~(2)]{Hirsch-Loc}. By
  Corollary~\ref{cor:holim_detects_colocal} this is equivalent to
  saying that the map $\holim f^c (-\vec k)$ is a quasi-isomorphism
  for all $\vec k \in R_\fan$. However, since $\holim$ and twisting
  both preserve weak equivalences, this is equivalent, in view of
  diagram~(\ref{eq:square_colocal}) above, to the condition that
  $\holim f(-\vec k)$ is a quasi-isomorphism for all $\vec k \in
  R_\fan$.
\end{proof}

\section{The derived category}

Our next goal is to prove that for a large class of schemes the
(unbounded) derived category of quasi-coherent sheaves can be obtained
as the homotopy category of homotopy sheaves.

The material in this section will apply to any regular toric
scheme~$X$ defined over a commutative ring~$A$; more generally, it
will be enough to assume that $X$ is a scheme equipped with a finite
semi-separating cover \cite[\S B.7]{TT-KofSchemes} as specified in
Definition~\ref{def:fan-coverings} below. Then the categories of chain
complexes of quasi-coherent sheaves on~$U_\sigma$ and $X$,
respectively, admit the {\it injective model structure\/} with
cofibrations the levelwise injective maps, and the categories of chain
complexes of quasi-coherent sheaves on~$U_\sigma$ admit the {\it
  projective model structure\/} with fibrations the levelwise
surjective maps. Finally, all the inclusions $U_\sigma \subseteq X$
are affine maps and hence induce exact push-forward functors.

\subsection{Coverings indexed by a fan}

\begin{definition}
  \label{def:fan-coverings}
  Let $A$ be a commutative ring, let $\fan$ be a finite fan in
  $N_\bR$, and let $X$ be an $A$-scheme.  A collection
  $(U_\sigma)_{\sigma \in \fan}$ of open subschemes of~$X$ is called a
  {\it $\fan$-covering\/} if $\bigcup_{\sigma \in \fan} U_\sigma = X$,
  and if for all $\tau, \sigma \in \fan$ we have $U_\tau \cap U_\sigma
  = U_{\tau \cap \sigma}$. If all the $U_\sigma$ are affine, we call
  $(U_\sigma)_{\sigma \in \fan}$ an {\it affine $\fan$-covering}.
\end{definition}

If the $A$-scheme~$X$ admits an affine $\fan$-covering, for some
finite fan~$\fan$, then $X$ is necessarily quasi-compact and
semi-separated \cite[\S B.7]{TT-KofSchemes}, hence in particular
quasi-separated. These facts are relevant as they guarantee the
existence of certain model category structures,
cf.~\S\ref{subsec:model_structures_sheaves}.

\begin{example}
  Every quasi-compact separated scheme~$X$ admits an affine
  $\fan$-covering for some fan~$\fan$. Indeed, let $U_0, U_1, \ldots,
  U_n$ be an open affine cover of~$X$. Let $\fan$ denote the usual fan
  of $n$-dimensional projective space, described as follows. Let $e_1,
  e_2, \ldots, e_n$ denote the unit vectors of~$\bR^n$, set $e_0 =
  -e_1-e_2- \ldots -e_n$, and define $M:= \{0,\, 1,\, \ldots,\,
  n\}$. Then $\fan$ is the collection of cones
  \[\sigma_E = \mathrm{cone} \Big(\{ e_i \,|\, i \in E\}\Big) \subset
  \bR^n\]
  for {\it proper\/} subsets $E \subset M$. Given such a set~$E$
  define $U_{\sigma_E} := \bigcap_{i \in M \setminus E} U_i$; these
  intersections are affine since $X$ is separated. Then
  $(U_\sigma)_{\sigma \in \fan}$ is an affine $\fan$-covering of~$X$
  by construction.

  More generally, if $X$ is quasi-compact, and the sets $U_0, U_1,
  \ldots, U_n$ form a semi-separating covering of~$X$, the above
  construction provides an affine $\fan$-covering for~$X$.
\end{example}

\subsection{Sheaves and homotopy sheaves}

From now on we will assume that $A$ is a commutative ring, that $\fan$
is a finite fan in~$N_\bR$, and that $X$ is an $A$-scheme equipped
with an affine $\fan$-covering $(U_\sigma)_{\sigma \in \fan}$
(Definition~\ref{def:fan-coverings}); {\it a fortifiori}, $X$ is
quasi-compact and semi-separated.

For any open subscheme $Y \subseteq X$ we write $\qco (Y)$ for the
category of quasi-coherent sheaves of $\mathcal{O}_Y$-modules, and
$\mathrm{Ch}\, \qco (Y)$ for the category of (possibly unbounded)
chain complexes in~$\qco(Y)$.---In what follows we will consider a
presheaf to have values in the categories $\mathrm{Ch}\, \qco
(U_\sigma)$ rather than in chain complexes of modules:

\begin{definition}
  The category $\varpre (\fan)$ of presheaves on~$X$ is the category
  of $\fan^\mathrm{op}$-diagrams $C$ which assign to each $\sigma \in
  \fan$ an object $C^\sigma \in \mathrm{Ch}\, \qco (U_\sigma)$, and to
  each inclusion $\tau \subseteq \sigma$ in $\fan$ a map
  $C^\sigma|_{U_\tau} \rTo C^\tau$, which is the identity for $\tau =
  \sigma$, subject to the condition that for $\nu \subseteq \tau
  \subseteq \sigma$ in~$\fan$ the composition
  \[C^\sigma|_{U_\nu} = \Big(C^\sigma|_{U_\tau}\Big)|_{U_\nu} \rTo
  (C^\tau)|_{U_\nu} \rTo C^\nu\]
  coincides with the structure map corresponding to the inclusion $\nu
  \subseteq \sigma$.
\end{definition}

The category $\varpre(\fan)$ is another example of a twisted
diagram category in the sense of~\cite[\S2.2]{HR-Twisted}, formed with
respect to the adjunction bundle
\[\fan^\op \rTo \mathrm{Cat}, \quad \sigma \mapsto \mathrm{Ch} \, \qco
(U_\sigma)\]
and structural adjunctions given by restriction (the left adjoints)
and push-forward along inclusions. We can thus appeal to the general
machinery of twisted diagrams again to equip $\varpre(\fan)$ with
various model structures.

\medbreak

We define the notions of strict and homotopy sheaves for $\varpre
(\fan)$ in analogy to Definition~\ref{def:preheaves}:

\begin{definition}
  \label{def:var_homotopy_sheaf}
  Given an object $C \in \varpre (\fan)$ we call $C$ a {\it strict
    sheaf\/} if for all inclusions $\tau \subseteq \sigma$ in~$\fan$
  the structure map $C^\sigma|_{U_\tau} \rTo C^\tau$ is an
  isomorphism; we call~$C$ a {\it homotopy sheaf\/} if for all
  inclusions $\tau \subseteq \sigma$ in~$\fan$ the structure map
  $C^\sigma|_{U_\tau} \rTo C^\tau$ is a quasi-isomorphism.
\end{definition}

\begin{remark}
  \label{rem:hom_invariant_alt}
  Since restriction to the open subset $U_\sigma$ is an exact functor,
  Lemmas~\ref{lem:homotopy_invariant}, \ref{lem:homotopy_sheaf_2_of_3}
  and~\ref{lem:hom_sheaf_inv_retracts} apply {\it mutatis mutandis}.
  That is, if $f \colon C \rTo D$ is a map in $\varpre(\fan)$ which is
  a quasi-isomorphism on each $U_\sigma$, we know that $C$ is a
  homotopy sheaf if and only if $D$ is a homotopy sheaf. Moreover, the
  class of homotopy sheaves is closed under kernels, cokernels,
  extensions, and retracts.
\end{remark}

\begin{remark}
  \label{rem:sheaves_modules_same}
  In the case of a toric scheme the categories $\pre(\fan)$
  (Definition~\ref{def:preheaves}) and $\varpre (\fan)$ codify the
  same information. Recall that for an affine scheme $U =
  \mathrm{Spec} (B)$ the category of quasi-coherent sheaves on~$U$ is
  equivalent, via the exact global sections functor, to the category
  of $B$-modules. Consequently, {\it if~$X=X_\fan$ is a regular
    toric scheme with fan~$\fan$, the functor
  \[\varpre(\fan) \rTo \pre(\fan), \quad C \mapsto \Big(
  \sigma \mapsto \Gamma(C; U_\sigma) \Big)\]
  is an equivalence of categories. It maps strict sheaves to
  strict sheaves, and homotopy sheaves to homotopy sheaves.\/}

The difference between $\pre (\fan)$ and $\varpre (\fan)$ is of a
purely technical nature; the choice of which category to use is mostly
dictated by convenience rather than necessity. Our previous results on
homotopy sheaves and colocalisation thus apply {\it mutatis mutandis}
for a regular toric scheme~$X_\fan$.
\end{remark}

\subsection{Model structures}
\label{subsec:model_structures_sheaves}

For every quasi-separated and quasi-compact scheme $Y$ the category
$\qco (Y)$ of quasi-coherent $\mathcal{O}_Y$-module sheaves is a
\textsc{Grothendieck\/} \textsc{abel}ian category \cite[\S
B.3]{TT-KofSchemes} which, in particular, satisfies axiom~AB5
(``filtered colimits are exact'').  It is well-known \cite{MR1814077}
that therefore the category $\mathrm{Ch}\, \qco (Y)$ of (possibly
unbounded) chain complexes of quasi-coherent sheaves on~$Y$ admits the
{\it injective model structure\/} with weak equivalences the
quasi-isomorphisms, and cofibrations the levelwise injections.

Since a semi-separated scheme is automatically quasi-separated, and
quasi-separatedness is stable under passage to open subschemes, this
applies to our scheme $X$ as well as to all the covering
sets~$U_\sigma$.

\bigbreak

The full subcategory of~$\varpre (\fan)$ spanned by the strict sheaves
is equivalent to the category $\mathrm{Ch}\,\qco(X)$ of (unbounded)
chain complexes of quasi-coherent sheaves on~$X_\fan$.  Its derived
category $D(\qco (X))$ can be obtained as the homotopy category of
the injective model structure of $\mathrm{Ch}\,\qco(X)$ described
above.

\begin{lemma}
  Let $U \subseteq X$ be an open subset.  The functor
  \[\mathrm{Ch}\,\qco(X) \rTo \mathrm{Ch}\,\qco(U), \quad
  \mathcal{F} \mapsto \mathcal{F}|_U\]
  is a left \textsc{Quillen\/} functor with right adjoint given by
  push-forward along the inclusion $U \rTo X$. (Here we equip
  $\mathrm{Ch}\,\qco(U)$ with the injective model structure as well.)
\end{lemma}

\begin{proof}
  This follows from the fact that restriction to open subsets is
  exact, hence preserves weak equivalences (quasi-isomorphisms) and
  cofibrations (injections).
\end{proof}

\begin{lemma}
  \label{lem:model_structure_varpre}
  The category $\varpre(\fan)$ has a model structure where a map is a
  weak equivalence if it is an objectwise quasi-isomorphism, and a
  cofibration if it is objectwise and levelwise injective.
\end{lemma}

\begin{proof}
  This is the $f$-structure of \cite[Theorem~3.3.5]{HR-Twisted}, based
  on the injective model structure of the categories $\mathrm{Ch}\,
  \qco (U_\sigma)$.
\end{proof}

Fibrations in this model structure can be characterised using {\it
  matching complexes\/}: Given $C \in \varpre (\fan)$ and $\sigma \in
\fan$ define $M^\sigma C = \lim_{\tau \subset \sigma} i^\tau_*
(C^\tau)$ where $i^\tau \colon U_\tau \subseteq U_\sigma$ is the
inclusion, and the limit is taken over all $\tau \in \fan$ strictly
contained in~$\sigma$. Then $f \colon C \rTo D$ is a fibration if and
only if for all $\sigma \in \fan$ the induced map
\begin{equation}
  \label{eq:fib_varpre}
  C^\sigma \rTo M^\sigma C \times_{M^\sigma D} D^\sigma
\end{equation}
is a fibration in the category $\mathrm{Ch}\,\qco(U_\sigma)$.---If $f$
is a fibration then in particular all the components $f^\sigma \colon
C^\sigma \rTo D^\sigma$ are fibrations in their respective categories.

\subsection{Strictifying homotopy sheaves}

Now consider the ``constant diagram'' functor, defined by
\[\Phi \colon \mathrm{Ch}\, \qco (X) \rTo \varpre (\fan), \quad
{\cal F} \mapsto \Big( \sigma \mapsto {\cal F}|_{U_\sigma} \Big) \ .\]

With respect to the model structure of
Lemma~\ref{lem:model_structure_varpre} the functor~$\Phi$ is left
\textsc{Quillen} (by exactness of restriction to open subsets) with
right adjoint given by
\[\Xi \colon \varpre (\fan) \rTo \mathrm{Ch}\, \qco (X), \quad C
\mapsto
\lim_{\sigma \in \fan^{\mathrm{op}}} j^\sigma_* (C^\sigma)\]
where the $j^\sigma \colon U_\sigma \rTo X$ are the various inclusion
maps. By construction we have canonical maps $\Xi(C) \rTo j^\sigma_*
C^\sigma$ which give rise, upon restriction to $U_\sigma$, to maps
\[r_\sigma \colon (\Xi(C))|_{U_\sigma} \rTo (j^\sigma_*
C^\sigma)|_{U_\sigma} = C^\sigma \ .\]
These maps are natural in~$\sigma$ in the sense that for each
inclusion $\tau \subseteq \sigma$ of cones in~$\fan$ the map $r_\tau$
equals the composite map
\begin{equation}
  \label{eq:naturality_of_r}
  \Xi(C)|_{U_\tau} = (\Xi(C)|_{U_\sigma})|_{U_\tau}
  \rTo[l>=3.5em]^{r_\sigma|_{U_\tau}} C^\sigma|_{U_\tau} \rTo C^\tau \ .
\end{equation}
In other words, the maps $r_\sigma$ assemble to a map of presheaves
\[r \colon \Phi \circ \Xi (C) \rTo C\]
which is the counit of the adjunction of~$\Phi$ and~$\Xi$.

\medbreak

Recall that an object $C \in \varpre (\fan)$ is a homotopy sheaf
(Definition~\ref{def:var_homotopy_sheaf}) if the structure maps
$C^\sigma|_{U_\tau} \rTo C^\tau$ are quasi-isomorphisms for all
inclusions $\tau \subseteq \sigma$ in~$\fan$. The following Lemma
shows how the functor~$\Xi$ can be used to strictify homotopy sheaves,
\ie, how to replace a homotopy sheaf by weakly equivalent strict
sheaf:

\begin{lemma}
  \label{lem:pullback_gluing}
  Every homotopy sheaf $\bar C \in \varpre(\fan)$ is weakly equivalent
  to a strict sheaf. More precisely, let $C \lTo^\simeq \bar C$ denote
  a fibrant replacement. Then for each $\sigma \in \fan$ the canonical
  map
  \[r_\sigma \colon \Xi(C)|_{U_\sigma} \rTo C^\sigma\]
  is a quasi-isomorphism in $\mathrm{Ch}\, \qco (U_\sigma)$. In other
  words, we have a chain of weak equivalences of homotopy sheaves
  \[\Phi \circ \Xi(C) \rTo^\simeq_r C \lTo^\simeq \bar C\]
  where $\Phi \circ \Xi(C)$ is, in fact, a strict sheaf.
\end{lemma}

\begin{proof}
  First note that $C$, being weakly equivalent to the homotopy
  sheaf~$\bar C$, is a homotopy sheaf by
  Remark~\ref{rem:hom_invariant_alt}.

  We have to prove that the map $r_\sigma \colon \Xi(C)|_{U_\sigma}
  \rTo C^\sigma$ is a weak equivalence in the category $\mathrm{Ch}\,
  \qco (U_\sigma)$.  In fact, it is enough to prove the claim for all
  maximal cones~$\sigma$: Given any $\tau \in \fan$ choose a maximal
  cone~$\sigma$ containing~$\tau$.  By~(\ref{eq:naturality_of_r}), the
  map~$r_\tau$ then is the composition of the restriction of the weak
  equivalence $r_\sigma$ to~$U_\tau$ with the structure map
  $C^\sigma|_{U_\tau} \rTo C^\tau$.  The latter is a quasi-isomorphism
  since $C$~is a homotopy sheaf, the former is a quasi-isomorphism
  since restriction is exact. Hence $r_\tau$~is a weak equivalence.

  So let $\sigma \in \fan$ be a maximal cone. We want to show that the
  top horizontal map $t = r_\sigma$ in the following diagram is a weak
  equivalence (where $j^{\tau} \colon U_{\tau} \rTo X$ denotes the
  inclusion map as before):
  \begin{equation}
    \label{diagram_h}
    \begin{diagram}
      \hbox to 0 pt {\hss \( (\Xi(C))|_{U_\sigma} \ = \ \)} \lim_{\tau
        \in \fan^\op} (j^\tau_* C^\tau)|_{U_\sigma} & \rTo^t &
      \lim_{\tau \subseteq \sigma} (j^\tau_* C^\tau)|_{U_\sigma} &
      \iso C^\sigma
      \\
      \dTo && \dTo>p \\
      \lim_{\tau \not= \sigma} (j^\tau_* C^\tau)|_{U_\sigma} &
      \rTo[l>=3em]^h & \lim_{\tau \subset \sigma} (j^\tau_*
      C^\tau)|_{U_\sigma}
    \end{diagram}
  \end{equation}

  The diagram is cartesian: It arises from first re-writing
  the limit defining $\Xi(C)$ as a pullback of limits indexed over
  smaller categories, then applying the exact restriction functor
  $(\,\cdot\,)|_{U_\sigma}$.  Moreover, the map~$p$ is a fibration
  since $C$ is a fibrant object; indeed, $p$~is nothing but the
  map~(\ref{eq:fib_varpre}) corresponding to $\sigma \in \fan$ for the
  map $C \rTo 0$. Hence by right properness of the injective model
  structure of $\mathrm{Ch}\, \qco (U_\sigma)$ it is enough to show
  that the lower horizontal map~$h$ is a weak equivalence.

  \medbreak

  For $\nu \subseteq \sigma$ let $i^\nu \colon U_\nu \rTo U_\sigma$
  and $j^\nu \colon U_\nu \rTo X$ denote the inclusions. Then we have
  an equality
  \begin{equation}
    \label{eq:res_1}
    \big(j^\nu_* (\mathcal{F})\big)|_{U_\sigma} = i^\nu_*
    (\mathcal{F}) \quad \hbox{for } \mathcal{F} \in \qco (U_\nu) \ ,
  \end{equation}
  and if $\tau \supseteq \nu$ is another cone,
  \begin{equation}
    \label{eq:res_2}
    \big( j^\tau_* (\mathcal{G})\big)|_{U_\sigma} = i^\nu_*
    (\mathcal{G})|_{U_\nu} \quad \hbox{for } \mathcal{G} \in \qco
    (U_\tau) \ .
  \end{equation}

  We embed the map~$h$ of diagram~(\ref{diagram_h}) above into the
  larger diagram (\ref{eq:huge_diagram}) below. We have
  used~(\ref{eq:res_1}) for the upper vertical map on the right,
  and~(\ref{eq:res_2}) for the upper vertical map on the left (recall
  also that restriction and push forward are exact functors, hence
  commute with finite limits). The map~$f$ is induced by the structure
  maps $C^\tau|_{U_{\tau \cap \sigma}} \rTo C^{\tau \cap \sigma}$
  of~$C$.

  \begin{equation}
  \label{eq:huge_diagram}
  \begin{diagram}
    \lim_{\tau \not= \sigma} (j^\tau_* C^\tau)|_{U_\sigma} & \rTo^h &
    \lim_{\tau \subset \sigma} (j^\tau_* C^\tau)|_{U_\sigma} \\
    \dTo<= && \dTo>= \\
    \lim_{\tau \not= \sigma} i^{\tau \cap \sigma}_* (C^\tau|_{U_{\tau
        \cap \sigma}}) & \rTo[l>=3em] & \lim_{\tau \subset \sigma}
    i^{\tau}_* (C^\tau) \\
    \dTo<\simeq>f && \dTo>= \\
    \lim_{\tau \not= \sigma} i^{\tau \cap \sigma}_* (C^{\tau \cap
      \sigma}) & \rTo[l>=5em]_g^\iso & \lim_{\tau
      \subset \sigma} i^{\tau}_* (C^\tau)
  \end{diagram}
  \end{equation}
  The map~$g$ is easily seen to be an isomorphism: In the
  diagram $\tau \mapsto i^{\tau \cap \sigma}_* (C^{\tau \cap \sigma})$
  all structure maps corresponding to the inclusions $\tau \cap \sigma
  \subseteq \tau$ are isomorphisms, hence all terms with $\tau
  \not\subseteq \sigma$ are redundant when forming the limit, and the
  map~$g$ is given by forgetting the redundant terms.

  \medbreak

  We are thus reduced to showing that the map~$f$ is a
  quasi-isomorphism which will follow from an application of
  \textsc{Brown}'s Lemma \cite[dual of Lemma~9.9]{Dwyer-Spalinski}.

  \medbreak
  
  We need some preliminary remarks. Recall that since $U_\sigma$ is
  affine, say $U_\sigma = \mathrm{Spec}\, A_\sigma$, the category
  $\mathrm{Ch}\, \qco (U_\sigma)$ is equivalent to the category of
  $A_\sigma$-modules. Hence $\mathrm{Ch}\, \qco (U_\sigma)$ is
  equivalent to the category $\mathrm{Ch}_{A_\sigma}$, which implies
  that we can equip the category $\mathrm{Ch}\, \qco (U_\sigma)$ with
  the {\it projective model structure\/}: Fibrations are the levelwise
  surjective maps, and weak equivalences are the quasi-isomorphisms. A
  cofibration in the projective model structure turns out to be
  levelwise injective (even levelwise split injective), but this
  condition does not characterise cofibrations.

  \medbreak

  We will
  denote the category of functors $(\fan \setminus
  \{\sigma\})^\op \rTo \mathrm{Ch}\, \qco (U_\sigma)$ by
  \[\mathcal{C} := \mathrm{Fun}\, \big((\fan \setminus
  \{\sigma\})^\op,\, \qco (U_\sigma)\big) \ .\]
  The category~$\mathcal{C}$ carries a model structure where a map is
  a weak equivalence (\resp, cofibration) if and only if it is an
  objectwise weak equivalence (\resp, cofibration in the projective
  model structure). A diagram $D \in \mathcal{C}$ is fibrant if and
  only if for all $\nu \in \fan \setminus \{\sigma\}$ the map
  \[D^\nu \rTo \lim_{\tau \subset \nu} D^\tau\]
  is a fibration in the projective model structure (\ie, is levelwise
  surjective), the limit taken over all cones $\tau \in \fan \setminus
  \sigma$ strictly contained in~$\nu$.

  With respect to the projective model structure of $\mathrm{Ch}\,
  \qco (U_\sigma)$ the inverse limit functor
  \[\lim \colon \mathcal{C} \rTo \mathrm{Ch}\, \qco (U_\sigma), \quad D
  \mapsto \lim_{\fan \setminus \{\sigma\})^\op} (D)\]
  is right \textsc{Quillen\/} with left adjoint given by the constant diagram
  functor
  \[\Delta \colon \mathrm{Ch}\, \qco (U_\sigma) \rTo \mathcal{C},
  \quad C \mapsto \Big( \Delta (C) \colon \tau \mapsto C \Big) \ ;\]
  note that $\Delta$ preserves weak equivalences and cofibrations as
  these notions are defined objectwise in~$\mathcal{C}$. Thus, using
  \textsc{Brown}'s Lemma \cite[dual of Lemma~9.9]{Dwyer-Spalinski}, we
  know that {\it if $f$ is a weak equivalence in~$\mathcal{C}$ with source
  and target fibrant diagrams, then $\lim (f)$ is a weak equivalence
  in $\mathrm{Ch}\, \qco(U_\sigma)$}.

  \medbreak

  We will apply this last observation to the map~$f$ in the
  diagram~(\ref{eq:huge_diagram}): We know that $f$~is a weak
  equivalence provided we can verify the following three assertions:
  \begin{itemize}
  \item [\rm (1)] The natural transformation of diagrams defining~$f$ consists
  of weak equivalences (quasi-isomorphisms)
  \item[\rm (2)] The diagram $\tau \mapsto i^{\tau \cap \sigma}_*
  (C^\tau|_{U_{\tau \cap \sigma}})$ (the source of~$f$) is a fibrant
  object of~$\mathcal{C}$
  \item[\rm (3)] The diagram $\tau \mapsto i^{\tau \cap \sigma}_*
  (C^{\tau \cap \sigma})$ (the target of~$f$) is a fibrant object
  of~$\mathcal{C}$
  \end{itemize}

  Assertion~(1) is easy to verify. The map~$f$ is induced by the structure
  maps $C^\tau|_{U_{\tau \cap \sigma}} \rTo C^{\tau \cap \sigma}$ which are weak
  equivalences since~$C$ is a homotopy sheaf by hypothesis?; note also that the
  functor $i_*^{\tau \cap \sigma}$ is exact since the inclusion $U_{\tau \cap
  \sigma} \subseteq U_\sigma$ is affine.

  \medbreak

  For assertion~(2) we have to verify that for each $\nu \in \fan
  \setminus \sigma$ the map
  \begin{equation}
    \label{eq:assertion_ii}
    i_*^{\nu \cap \sigma}
    (C^\nu|_{U_{\nu \cap \sigma}}) \rTo \lim_{\tau 
    \subset \nu} i_*^{\tau \cap \sigma} (C^\tau|_{U_{\tau \cap \sigma}})
  \end{equation}
  is levelwise surjective. By hypothesis $C$ is a fibrant object
  (Lemma~\ref{lem:model_structure_varpre}) of~$\varpre (\fan)$, so
  the map
  \[C^\nu \rTo \lim_{\tau \subset \nu} k_*^\tau (C^\tau)\]
  (with $k^\tau$ being the inclusion $U_\tau \subseteq U_\nu$) is a fibration
  in the injective model structure of $\mathrm{Ch}\, \qco (U_\nu)$; in
  particular, this map is levelwise surjective. Since restriction to open
  subsets is exact, it follows that the map
  \[C^\nu|_{U_{\nu \cap \sigma}} \rTo \lim_{\tau \subset \nu} (k_*^\tau
  (C^\tau))|_{\nu \cap \sigma} = \lim_{\tau \subset \nu} \ell^{\tau \cap
    \sigma}_* (C^\tau|_{U_{\tau \cap \sigma}})\]
  is levelwise surjective, where now $\ell^{\tau \cap \sigma}$ denotes the
  inclusion $U_{\tau \cap \sigma} \subseteq U_{\nu \cap \sigma}$. We can now
  apply the exact functor $i_*^{\nu \cap \sigma}$; since $i_*^{\nu \cap
    \sigma} \circ \ell_*^{\tau \cap \sigma} = i_*^{\tau \cap \sigma}$ we
  conclude that the map~(\ref{eq:assertion_ii}) is levelwise surjective as
  claimed.
  
  \medbreak

  We now discuss assertion~(3). We have to show that for each $\nu
  \in \fan \setminus \sigma$ the map
  \begin{equation}
    \label{eq:assertion_iii}
    i_*^{\nu \cap \sigma}
    (C^{\nu \cap \sigma}) \rTo \lim_{\tau 
    \subset \nu} i_*^{\tau \cap \sigma} (C^{\tau \cap \sigma})
  \end{equation}
  is levelwise surjective (where $i^\mu \colon U_\mu \rTo U_\sigma$ as
  before).

  \medbreak
  
  Consider the diagram
  \[D \colon \{\tau \subset \nu\}^\op \rTo \mathrm{Ch}\, \qco
  (U_\sigma),\quad \tau \mapsto i^{\tau \cap \sigma}_* (C^{\tau \cap
    \sigma}) \ ,\]
  its limit being the target of the map~(\ref{eq:assertion_iii}). If
  $\nu \subset \sigma$ then the map~(\ref{eq:assertion_iii}) arises by
  application of the exact functor $i_*^\nu = i_*^{\nu \cap \sigma}$
  to the map
  \begin{equation}
    \label{eq:assertion_iii_2}
    C^\nu = C^{\nu \cap \sigma} \rTo \lim_{\tau \subset \nu} \ell^\tau_* C^\tau
  \end{equation}
  where $\ell^\tau \colon U_\tau \rTo U_\nu = U_{\nu \cap
    \sigma}$ is the inclusion map. Now $C$ is a fibrant object
  of~$\varpre(\fan)$ by hypothesis, so~(\ref{eq:assertion_iii_2}) is
  a fibration in the injective model structure, hence levelwise
  surjective. It follows that~(\ref{eq:assertion_iii}) is levelwise
  surjective as well.

  \medbreak

  It remains to deal with the case $\nu \not \subseteq \sigma$. Let
  $\tau~$ be a proper face of~$\nu$.  The structure maps of~$D$
  corresponding to the inclusions $\tau \cap \sigma \subseteq \tau$
  are identity maps:
  \[i^{\tau \cap \sigma}_* (C^{\tau \cap \sigma}) = i^{(\tau \cap
    \sigma) \cap \sigma}_* (C^{(\tau \cap \sigma) \cap \sigma}) \rTo
  i^{\tau \cap \sigma}_* (C^{\tau \cap \sigma})\]
  It follows that the limit of~$D$ is isomorphic to the limit of the
  restriction of~$D$ to faces of the form $\tau \cap \sigma$ for $\tau
  \subset \nu$. So define $Q:=  \{\tau \cap \sigma \,|\, \tau \subset
  \nu \}$. In fact, $Q$ is the poset of proper faces of~$\nu$ which
  are also faces of~$\sigma$. Now since $\nu \not \subseteq \sigma$ we
  know that $Q$ has maximal element $\nu \cap \sigma \subset \nu$.
  With this notation, the map~(\ref{eq:assertion_iii}) can be embedded
  into a chain
  \[i_*^{\nu \cap \sigma} (C^{\nu \cap \sigma})
  \rTo^{~(\ref{eq:assertion_iii})} \lim_{\tau \subset \nu} i_*^{\tau
    \cap \sigma} (C^{\tau \cap \sigma}) \rTo^\iso \lim_{\tau \in
    Q^\op} i_*^{\tau \cap \sigma} (C^{\tau \cap \sigma}) \iso i_*^{\nu
    \cap \sigma} (C^{\nu \cap \sigma})\]
  with composition the identity map. It follows that the
  map~(\ref{eq:assertion_iii}) is levelwise surjective as claimed.
\end{proof}

\subsection{The derived category via homotopy sheaves}

We have constructed a pair of adjoint functors
\[\Phi \colon \mathrm{Ch}\, \qco (X) \rTo \varpre (\fan)
\quad\hbox{and}\quad \Xi \colon \varpre (\fan) \rTo \mathrm{Ch}\, \qco
(X) \ ,\]
the functor $\Phi$ being the left adjoint. Moreover, the pair $(\Phi,
\Xi)$ is a \textsc{Quillen} pair with respect to the injective model
structure on~$\mathrm{Ch}\, \qco (X)$, and the model structure
described in Lemma~\ref{lem:model_structure_varpre}
on~$\varpre(\fan)$. From general model category theory, we obtain an
adjoint pair of total derived functors
\[\mathbf{L}\Phi \colon \mathrm{Ho}\, \mathrm{Ch}\, \qco (X) \rTo
\mathrm{Ho}\, \varpre (\fan) \ \hbox{and}\ \mathbf{R}\Xi \colon
\mathrm{Ho}\,\varpre (\fan) \rTo \mathrm{Ho}\,\mathrm{Ch}\,\qco(X)\]
which we can use to give a description of the derived category
$D\big(\qco (X)\big) = \mathrm{Ho}\, \mathrm{Ch}\, \qco (X)$ via
homotopy sheaves:

\begin{theorem}
  \label{thm:derived_is_correct}
  Let $\mathcal{H}$ denote the full subcategory of $\mathrm{Ho}\,
  \varpre (\fan)$ spanned by the homotopy sheaves. The
  \textsc{Quillen} pair $(\Phi, \Xi)$ induces an equivalence of
  categories
  \[\mathbf{L}\Phi \colon \mathrm{Ho}\, \mathrm{Ch}\, \qco (X) \rTo
  \mathcal{H}\]
  with inverse given by~$\mathbf{R}\Xi$.
\end{theorem}

\begin{proof}
  We first have to verify that $\mathbf{L}\Phi$ takes values
  in~$\mathcal{H}$. Every object~$\mathcal{F}$ of
  $\mathrm{Ch}\, \qco (X)$ is cofibrant in the injective model
  structure, hence $\mathbf{L}\Phi (\mathcal{F}) \iso \Phi
  (\mathcal{F})$ in $\mathrm{Ho}\, \varpre (\fan)$, and the relevant
  structure maps
  \[\Phi(\mathcal{F})^\sigma|_{U_\tau} =
  (\mathcal{F}|_{U_\sigma})|_{U_\tau} = \mathcal{F}|_{U_\tau} = \Phi
  (\mathcal{F})^\tau\]
  are identities, hence weak equivalences. This shows that
  $\mathbf{L}\Phi (\mathcal{F})$ is a homotopy sheaf, so
  $\mathbf{L}\Phi (\mathcal{F}) \in \mathcal{H}$.

  Given $C \in \mathcal{H}$ the counit map of
  the adjunction of $\mathbf{L}\Phi$ and $\mathbf{R}\Xi$ is modelled by
  the point-set level counit map of $(\Phi, \Psi)$ at $C^\mathrm{f}$,
  \[\epsilon_{C^\mathrm{f}} \colon \Phi (\Xi (C^\mathrm{f})) \rTo
  C^\mathrm{f}\]
  where $C \rTo^\sim C^\mathrm{f}$ denotes a fibrant replacement in
  $\varpre(\fan)$. Fix a cone $\sigma \in \fan$. The $\sigma$-component
  of~$\epsilon_{C^\mathrm{f}}$ is nothing but the map~$r_\sigma$ of
  Lemma~\ref{lem:pullback_gluing} applied to~$C^\mathrm{f}$.
  Since~$C^\mathrm{f}$ is a homotopy sheaf
  Lemma~\ref{lem:pullback_gluing} applies, and we conclude that
  $\epsilon_{C^\mathrm{f}}$ is a weak equivalence. Hence $\mathbf{L}
  \Phi \circ \mathbf{R} \Xi (C) \rTo C$ is an isomorphism
  in~$\mathcal{H}$.

  Given $\mathcal{F} \in \mathrm{Ch}\, \qco (X)$ the unit map of
  the adjunction of $\mathbf{L}\Phi$ and $\mathbf{R}\Xi$ is modelled by
  the composition
  \begin{equation}
    \label{eq:unit}
    \mathcal{F} \rTo[l>=3em]^{\eta_\mathcal{F}} \Xi (\Phi(\mathcal{F}))
    \rTo[l>=3em]^{\Xi(a)} \Xi (\Phi(\mathcal{F})^\mathrm{f})
  \end{equation}
  where $a \colon \Phi(\mathcal{F}) \rTo^\sim
  \Phi(\mathcal{F})^\mathrm{f}$ denotes a fibrant replacement
  of~$\Phi(\mathcal{F})$ in~$\varpre(\fan)$, and
  where $\eta_\mathcal{F}$ is the point-set level adjunction unit of
  $(\Phi, \Xi)$. Since the functor~$\Phi$ detects weak equivalences it
  is enough to show that the composition of the two top horizontal
  maps in the following diagram is a weak equivalence:
  \begin{diagram}
    \Phi(\mathcal{F}) & \rTo[l>=4em]^{\Phi(\eta_\mathcal{F})} & \Phi
    (\Xi (\Phi(\mathcal{F}))) & \rTo[l>=5em]^{\Phi \circ \Xi (a)} &
    \Phi(\Xi (\Phi(\mathcal{F})^\mathrm{f})) \\
    & \rdTo<= & \dTo>{\epsilon_{\Phi(\mathcal{F})}} &&
    \dTo>\sim<{\epsilon_{\Phi(\mathcal{F})^\mathrm{f}}} \\
    && \Phi (\mathcal{F}) & \rTo_a^\sim & \Phi(\mathcal{F})^\mathrm{f}
  \end{diagram}
  The vertical maps are point-set level counit maps for
  $\Phi(\mathcal{F})$ and $\Phi(\mathcal{F})^\mathrm{f}$,
  respectively; hence the square commutes by naturality. The
  right-hand vertical map is a weak equivalence by
  Lemma~\ref{lem:pullback_gluing}, applied to the fibrant homotopy
  sheaf $\Phi(\mathcal{F})^\mathrm{f}$. The map~$a$ is the
  fibrant-replacement map, hence a weak equivalence, and the diagonal
  map is the identity by the theory of adjunctions (triangle
  identities \cite[\S{}IV, p.~83]{ML-working}).  This proves that the
  composition~(\ref{eq:unit}) is a weak equivalence as claimed.

  We have shown that both unit and counit maps of the adjunction
  $(\mathbf{L}\Phi, \mathbf{R}\Xi)$ are isomorphisms in the homotopy
  categories in question. Hence they give an equivalence of categories
  of $D\big(\qco (X)\big) = \mathrm{Ho} \, \mathrm{Ch}\, \qco (X)$ and
  $\mathcal{H}$ as claimed.
\end{proof}

\subsection{The derived category of a regular toric scheme}

\begin{theorem}
  \label{thm:D_of_X_from_diagrams}
  Let $A$ be a commutative ring with unit. Suppose that $\fan$ is a
  regular fan, and denote the associated $A$-scheme by~$X_\fan$. Let
  $R_\fan \subset \bZ^{\fan(1)} $ denote the finite set of integral
  vectors as specified in Construction~\ref{constr:R_fan}. The
  derived category $D(\qco(X_\fan))$ can be obtained from the twisted
  diagram category $\pre(\fan)$ defined in \ref{def:preheaves} by
  inverting all those maps $X \rTo Y$ which induce quasi-isomorphisms
  \begin{equation}
    \label{eq:weq_explicit}
    \holim\, X(-\vec k) \rTo^\sim \holim\, Y(-\vec k) \quad
    \hbox{ for all } \vec k \in R_\fan \ .
  \end{equation}

  More precisely, the homotopy category of the colocal model structure
  as described in Proposition~\ref{prop:colocal_model_structure} is
  equivalent to~$D(\qco(X_\fan))$. With respect to this model
  structure, the cofibrant objects are precisely the $c$-cofibrant
  homotopy sheaves, and a map of cofibrant objects is an objectwise
  weak equivalence if and only if it satisfies the
  condition~(\ref{eq:weq_explicit}).
\end{theorem}

\begin{proof}
  The characterisations of cofibrant objects and their colocal
  equivalences are given in
  Proposition~\ref{prop:colocal_model_structure} and
  Corollary~\ref{cor:weqs_of_homotopy_sheaves}. The homotopy category
  of the colocal model structure is equivalent to its
  subcategory~$\mathcal{A}$ spanned by homotopy sheaves (since every
  homotopy sheaves is isomorphic, via $c$-cofibrant replacement, to a
  colocal object). The category~$\mathcal{A}$ is equivalent to the
  subcategory~$\mathcal{H}$ of $\mathrm{Ho}\, \varpre(\fan)$ spanned
  by the homotopy sheaves,
  cf.~Remark~\ref{rem:sheaves_modules_same}. The
  category~$\mathcal{H}$, in turn, is equivalent to $D(\qco(X_\fan))$
  according to Theorem~\ref{thm:derived_is_correct}. This finished the
  proof.
\end{proof}

\begin{corollary}
  \label{cor:weak_generators_of_D}
  In the situation of Theorem~\ref{thm:D_of_X_from_diagrams}, the
  diagrams
  \[\mathcal{O}(\vec k), \ \vec k \in R_\fan\]
  form a set of weak generators of~$D(\qco(X_\fan))$: A morphism $f
  \colon C \rTo D$ in the category $D(\qco(X_\fan))$ is an isomorphism
  if and only if for all $\vec k \in R_\fan$ and all $\ell \in \bZ$,
  the map
  \[\hom (\mathcal{O}(\vec k)[\ell],\, f) \colon
  \hom (\mathcal{O}(\vec k)[\ell],\, C) \rTo^{f_*} \hom
  (\mathcal{O}(\vec k)[\ell], \, D)\]
  is an isomorphism of \textsc{abel}ian groups. Here $\mathcal{O}(\vec
  k)[\ell]$ denotes the diagram $\mathcal{O}(\vec k)$ considered as a
  chain complex concentrated in degree~$\ell$.
\end{corollary}

\begin{proof}
  By Theorem~\ref{thm:D_of_X_from_diagrams} it is enough to prove the
  corresponding statement for the homotopy category of the colocal
  model structure on~$\pre(\fan)$,
  cf.~Proposition~\ref{prop:colocal_model_structure}. Moreover,
  replacing $C$ by a cofibrant object we may
  assume that $f$ is represented by an actual map $g \colon C \rTo D$
  in~$\pre(\fan)$. The morphism~$f$ is an isomorphism if and only if
  $g$~is an $R_\fan$\nobreakdash-colocal equivalence.

  Morphism sets in the homotopy category can be described as the set
  of path components of mapping spaces; we are thus reduced to showing
  that $g$ is an $R_\fan$\nobreakdash-colocal equivalence if and only
  if the map
  \[\hom_{\pre(\fan)} \big(\hco(\vec k)[\ell] \tensor NA
  [\Delta^\bullet], \, C\big) \rTo^{g_*} \hom_{\pre(\fan)}
  \big(\hco(\vec k)[\ell] \tensor NA [\Delta^\bullet], \, D\big)\]
  induces a bijection after application of the functor~$\pi_0$ for all
  $\ell \in \bZ$ and all $\vec k \in R_\fan$. However, it follows from
  the proof of Proposition~\ref{prop:holim_detects_colocal} that $g_*$
  is a $\pi_0$-isomorphism if and only if the map
  \[\holim\, C(-\vec k) \rTo[PS] \holim\, D(-\vec k)\]
  is an $H_\ell$-isomorphism. This finishes the proof in view of
  Corollary~\ref{cor:holim_detects_colocal}
\end{proof}

\medbreak

In the special case of projective $n$-space the fan $\fan$ has $n+1$
different $1$\nobreakdash-cones. The set $R_\fan \subset \bZ^{n+1}$ as
defined in Construction~\ref{constr:R_fan} then consists of all the
possible $(0,1)$-vectors with at most $n$ non-zero entries,
cf.~Example~\ref{example:complete_fan}, and for any $\vec k \in
\bZ^{n+1}$ the line bundle $\mathcal{O}(\vec k)$ is isomorphic to the
line bundle usually denoted $\mathcal{O}_{\mathbb{P}^n}(\ell)$ where
$\ell = |\vec k|$ is the sum of the entries of~$\vec k$. In other
words, we recover the classical results that the sheaves
$\mathcal{O}_{\mathbb{P}^n}(\ell)$, $0 \leq \ell \leq n$, generate
the derived category. Note that Construction~\ref{constr:R_fan} gives
an explicit algorithm to construct generators for the derived category
of {\it any\/} regular toric scheme, defined over an arbitrary
commutative ring~$A$.

\goodbreak

\small
\raggedright

%BiBTeX stuff
\bibliographystyle{alpha}
\bibliography{colocal}

\end{document}